\documentclass[12pt]{article}
\usepackage{amssymb, color}
\usepackage{amsthm}
\usepackage{amsmath}
\usepackage{soul}
\usepackage{graphicx}
\usepackage{enumerate}
\usepackage{soul}

\DeclareMathOperator{\Con}{Con}

\DeclareMathOperator{\Fil}{Fil}

\newtheorem{theorem}{Theorem}[section]
\newtheorem{definition}[theorem]{Definition}
\newtheorem{lemma}[theorem]{Lemma}
\newtheorem{proposition}[theorem]{Proposition}
\newtheorem{remark}[theorem]{Remark}
\newtheorem{example}[theorem]{Example}
\newtheorem{corollary}[theorem]{Corollary}

\title{An algebraic analysis of implication in non-distributive logics}
\author{Ivan~Chajda, Kadir~Emir, Davide~Fazio, Helmut~L\"anger, \\
	Antonio~Ledda and Jan~Paseka}
\begin{document}
	
	\maketitle
	
	\begin{abstract}
		In this paper, we introduce the concept of a (lattice) skew Hilbert algebra as a natural generalization of Hilbert algebras. This notion allows a unified treatment of several structures of prominent importance for mathematical logic, e.g.\ (generalized) orthomodular lattices, and MV-algebras, which admit a natural notion of implication. In fact,  it turns out that skew Hilbert algebras play a similar role for (strongly) sectionally pseudocomplemented posets as Hilbert algebras do for relatively pseudocomplemented ones. We will discuss basic properties of closed, dense, and weakly dense elements of skew Hilbert algebras, their applications, and we will provide some basic results on their structure theory.
		
		%In this paper, we introduce the concept of (lattice) skew Hilbert algebra as a natural generalization of Hilbert algebras. These structures allow the treatment of several algebras of prominent importance for mathematical logic, like e.g.\ (generalized) orthomodular lattices, and MV-algebras, which admit a suitable notion of implication. It will turn out that Skew Hilbert algebras play a similar role for (strongly) sectionally pseudocomplemented posets as Hilbert algebras do for relatively pseudocomplemented ones. We will discuss basic properties of closed, dense, and weakly dense elements of skew Hilbert algebras, and we will provide some basic results on their structure theory.
	\end{abstract}
	
	
	
	\section{Introduction}
	
	D. Hilbert \cite{HI01} was the first to single out the importance of the implicative fragment of classical logic, namely the calculus obtained from classical propositional logic by assuming as axioms a given set of classical tautologies containing just the connective implication. This logical system, later called \emph{propositional calculus of positive implication} \cite{HI02}, revealed to be amenable of smooth algebraic investigations by means of Henkin's implicative models \cite{HE01}. Their duals,  known as Hilbert algebras after A. Diego's works on the topic \cite{D66a, D66b}, have been the subject of intensive and incredibly deep inquiries over the past years, see e.g.\ \cite{BUS,BUS02,Ce01}.\\ 
	In this article, we aim at generalizing the concept of a Hilbert algebra to a context in which the ``distributivity of implication over itself'' ((H5) in Definition \ref{def:HA}) need not hold in general. Indeed, Hilbert algebras provide the equivalent algebraic semantics in the sense of Blok and Pigozzi \cite{BlPi} to the implication fragment of intuitionistic propositional logic. Therefore, if we stick to the original notion, a relevant deal of prominent structures for algebraic logic, e.g.\ generalized orthomodular lattices, orthomodular lattices and MV-algebras, that indeed admit a natural notion of implication, are left out. It seems therefore quite natural to try to extend the very notion of Hilbert algebra to a wider framework, that may lend to encompass under a general umbrella those structures. We will see that skew Hilbert algebras play a similar role for (strongly) sectionally pseudocomplemented posets, i.e.\ posets with a top element and in which every principal order filter is a pseudocomplemented poset,  as Hilbert algebras do for relatively pseudocomplemented ones. In particular, we highlight the connections between skew Hilbert algebras and orthomodular implication algebras \cite{CL01}. Precisely, we will show that orthomodular implication algebras are indeed (term equivalent to) a subvariety of skew Hilbert algebras, axiomatized by two further identities. Then, making use of this result, we axiomatize the class of generalized orthomodular lattices within the class of skew Hilbert algebras, introduced by Janowitz in \cite{Ja68} (see also \cite{Be85}), and we show that they form in fact a variety. As a consequence, we obtain a characterization of orthomodular lattices in the framework of skew Hilbert algebras. Subsequently, we show that (strong) skew Hilbert algebras can be regarded as proper generalizations of lattices with sectional antitone involutions (see \cite{CHK01}) representing a common generalization of orthmodular lattices and MV-algebras. As a consequence, we will be able to frame algebras arising from seemingly different contexts into a common, general landscape.\\ 
	Let us now summarize the discourse of the paper. In Section \ref{sec: basic} we dispatch all the necessary preliminaries on Hilbert algebras, bounded posets with operations, and (sectionally) pseudocomplemented posets. In Section \ref{sec: SHA}, we discuss the concept of (lattice) skew Hilbert algebra and we provide some examples thereof. Section \ref{sec:app} is devoted to show how some structures having an underlying poset with sectional antitone operations, like e.g.\ orthomodular implication algebras, and lattices with sectional antitone involutions, can be framed within the theory of skew Hilbert algebras. In Section \ref{sec: specele}, we will describe basic properties of closed, dense, and weakly dense elements of skew Hilbert algebras. Finally, in Section \ref{sec: structheory} we will investigate the structure theory for the variety of lattice skew Hilbert algebras. Subsequently, we will introduce the concept of a deductive system on a skew Hilbert algebra. A full characterization thereof will follow. Finally, we introduce a notion of ``order-compatible'' congruence for skew Hilbert algebras which need not be lattice-ordered. In turn, we show that, also in this case, many of the aforementioned results hold.

	\section{Basic concepts}\label{sec: basic}
	
	The concept of Hilbert algebra was introduced by A.~Diego (see \cite{D66a, D66b}) and studied intensively by S.~Rudeanu (see \cite{R85, R07}). Let us recall its definition.
	\begin{definition}\label{def:HA}
		A {\em Hilbert algebra} is an algebra $(A,*,1)$ of type $(2,0)$ satisfying the following identities and quasi-identities for all $x,y,z\in A$:
		\begin{enumerate}[{\rm(H1)}]
			\item $x*x\approx1$,
			\item if $x*y=y*x=1$ then $x=y$,
			\item if $x*y=y*z=1$ then $x*z=1$,
			\item $x*(y*x)\approx1$,
			\item $(x*(y*z))*((x*y)*(x*z))\approx1$.
		\end{enumerate}
	\end{definition}

	Because of (H1) -- (H3) the binary relation $\leq$ on A defined by 
	\[
	x\leq y \text{ if and only if }x*y=1 \,\, (x,y\in A)
	\]
	is a partial order relation on $A$. From (H1) and (H4) we conclude $x*1\approx x*(x*x)\approx1$, thus $1$ is the top element of $(A,\leq)$. Hence, a Hilbert algebra can be alternatively defined as a poset $(A,\leq,*,1)$ with top element $1$ and with a binary operation $*$ satisfying the following conditions:
	\begin{itemize}
		\item $x\leq y$ if and only if $x*y=1$,
		\item $x\leq y*x$,
		\item $x*(y*z)\leq(x*y)*(x*z)$,
	\end{itemize}
	see e.g.\ \cite{D66a, D66b, R85, R07}.

	Let us now recall several useful concepts from the theory of posets. Other concepts used in this paper are taken from monographs \cite{B, G}.
	
	Let $(A,\leq)$ be a poset. An element $c$ of $A$ is called the {\em relative pseudocomplement} of $a$ with respect to $b$, in symbols $c=a\circ b$, if $c$ is the greatest element $x$ of $A$ satisfying $L(a,x)\subseteq L(b)$. If $x\circ y$ exists for all $x,y\in A$ then the poset $(A,\leq)$ is called {\em relatively pseudocomplemented} and $\circ$ is called {\em relative pseudocomplementation}.
	
	Relatively pseudocomplemented posets were investigated by three of the present authors in \cite{CLP20}. Relatively pseudocomplemented posets which are meet-semilattices are often called {\em implicative semilattices} or {\em Brouwerian semilattices}.
	
	It was shown by S.~Rudeanu \cite{R85} that the class of relatively pseudocomplemented posets is a proper subclass of the class of Hilbert algebras. In fact, a Hilbert algebra $(A,*,1)$ is a relatively pseudocomplemented poset if and only if for all $x,y\in A$, $x*y$ is the relative pseudocomplement of $x$ with respect to $y$.
	
	If a relatively pseudocomplemented poset $(A,\leq)$ is a lattice then it is called a {\em relatively pseudocomplemented lattice}, see \cite B. In such a case, for all $x,y\in A$, $x\circ y$ is the greatest element $z$ of $A$ satisfying $x\land z\leq y$.
	
	It is well-known that every relatively pseudocomplemented lattice is distributive, see e.g.\ \cite G. In order to extend relative pseudocomplementation to non-distributive lattices, the first author introduced in \cite C so-called {\em sectionally pseudocomplemented lattices}. Recall that a lattice $(L,\lor,\land)$ is called {\em sectionally pseudocomplemented} if every of its intervals $[y)$ is pseudocomplemented, or formally, if for every $a,b\in L$ there exists a greatest element $c$ of $L$ satisfying $(a\lor b)\land c=b$. This element $c$ is called the {\em sectional pseudocomplement} of $a$ with respect to $b$ and will be denoted by $a*b$.
	
	Of course, every relatively pseudocomplemented lattice is sectionally pseu\-do\-com\-ple\-men\-ted but, for example, the five-element non-modular lattice $\mathbf N_5$ is sectionally pseudocomplemented but not relatively pseudocomplemented, see \cite{C,CLPa} for examples and details.
	
	Let $(P,\leq)$ be a poset, $a,b\in P$ and $A,B\subseteq P$. We define the lower and upper cone of $A$ as follows:
	\begin{align*}
		L(A) & :=\{x\in P\mid x\leq A\}, \\
		U(A) & :=\{x\in P\mid x\geq A\}.
	\end{align*}
	Here $x\leq A$ means $x\leq y$ for all $y\in A$ and, similarly, $x\geq A$ means $x\geq y$ for all $y\in A$. The expression $A\leq B$ means $x\leq y$ for all $x\in A$ and $y\in B$. Instead of $L(\{a,b\})$ we simply write $L(a,b)$.
	
	The concept of a sectionally pseudocomplemented lattice was generalized to posets in \cite V as follows: Let $(A,\leq)$ be a poset and $a,b\in A$. An element $c$ of $A$ is called the {\em sectional pseudocomplement} of $a$ with respect to $b$ if it is the greatest element $x$ of $A$ satisfying $L(U(a,b),x)=L(b)$. This element $c$ will be denoted by $a*b$. A poset $(A,\leq)$ is called {\em sectionally pseudocomplemented} if for all $x,y\in A$ there exists $x*y$. Of course, in the case of lattices this concept coincides with the above one introduced for lattices. A unary operation $'$ on $A$ is called
	\begin{itemize}
		\item {\em antitone} if $x\leq y$ implies $y'\leq x'$,
		\item an {\em involution} if $x''\approx x$,
	\end{itemize}
	for all $x,y\in A$. A unary operation $'$ on a bounded poset $(A,\leq,0,1)$ is called
	\begin{itemize}
		\item a {\em complementation} if $L(x,x')=\{0\}$ and $U(x,x')=\{1\}$,
	\end{itemize}
	for all $x\in A$. An {\em orthoposet} is a bounded poset $(A,\leq,{}',0,1)$ with an antitone involution $'$ which is a complementation. An {\em ortholattice} is a lattice which is an orthoposet.
	
	Let us recall the following result from \cite{CLPa}:
	
	\begin{proposition}\label{prop1}
		The class of sectionally pseudocomplemented lattices forms a variety which is determined by the lattice axioms and the following identities:
		\begin{itemize}
			\item $z\lor y\leq x*((x\lor y)\land(z\lor y))$,
			\item $(x\lor y)\land(x*y)\approx y$.
		\end{itemize}
	\end{proposition}
	
	Also the next two results were proved in \cite{CLPa}.
	
	\begin{proposition}\label{prop2}
		Let $(P,\leq,*,1)$ be a sectionally pseudocomplemented poset with top element $1$. Then the following hold for all $x,y,z\in P$:
		\begin{enumerate}[{\rm(i)}]
			\item $x\leq y$ if and only if $x*y=1$,
			\item $1*x\approx x$,
			\item $x*(y*x)\approx1$,
			\item if $y*x=1$ then $x*((x*y)*y)=1$,
			\item if $x*y=1$ then $(y*z)*(x*z)=1$.
		\end{enumerate}
	\end{proposition}
	
	\begin{remark}
		By {\rm(i)} and {\rm(iii)} we derive $x\leq y*x$ and hence also $x\leq(y*x)*x$. Therefore, if $x\leq y$ then $x\leq y\leq(x*y)*y$ whence $x\leq(x*y)*y$. By {\rm(iv)} we have that $x\leq(x*y)*y$ provided $x$ and $y$ are comparable with each other. In order to avoid this rather restrictive condition, we define: A sectionally pseudocomplemented poset $(P,\leq,*,1)$ with top element $1$ is called a {\em strongly sectionally pseudocomplemented poset} if it satisfies the identity for all $x,y \in P$:
		\begin{enumerate}
			\item[{\rm(vi)}] $x\leq(x*y)*y$.
		\end{enumerate}
	\end{remark}
	
	\begin{proposition}\label{prop3}
		An algebra $(P,*,1)$ of type $(2,0)$ can be organized into a sectionally pseudocomplemented poset if and only if it satisfies
		\begin{itemize}
			\item $x*x\approx x*1\approx1$,
			\item $x*y=y*x=1$ implies $x=y$,
			\item $x*y=y*z=1$ implies $x*z=1$,
			\item $L(U(x,y),x*y)\approx L(y)$,
			\item $L(U(x,y),z)=L(y)$ implies $z*(x*y)=1$.
		\end{itemize}
		The last two conditions are formulated with respect to the partial order relation $\leq$ defined by $x\leq y$ if and only if $x*y=1$ {\rm(}$x,y\in P${\rm)}.
	\end{proposition}
	
	\section{Skew Hilbert algebras}\label{sec: SHA}
	
	As mentioned above, the class of relatively pseudocomplemented posets is a proper subclass of the class of Hilbert algebras. The aim of this section is to discuss the concept of {\em skew Hilbert algebras}. We will see in Section \ref{sec:app} that skew Hilbert algebras play a similar role with respect to orthomodular lattices and MV-algebras as Hilbert algebras do for the implicative fragment of Intuitionistic Propositional Logic.
	
	%As it will be clear, these structures play a similar role for (strongly) sectionally pseudocomplemented posets as Hilbert algebras do for relatively pseudocomplemented ones.
	
	\begin{definition}\label{def3}
		A {\em skew Hilbert algebra} is a poset $\mathbf S=(S,\leq,*,1)$ with a binary operation $*$ and a constant $1$ satisfying the following conditions:
		\begin{enumerate}[{\rm(S1)}]
			\item $x\leq y$ if and only if $x*y=1$,
			\item if $y*x=1$ then $x*((x*y)*y)=1$,
			\item if $x*y=1$ then $(y*z)*(x*z)=1$,
			\item $L(U(x,y),x*y)=L(y)$.
		\end{enumerate}
		If $\mathbf S$ satisfies the identity
		\begin{enumerate}
			\item[{\rm(S2')}] $x*((x*y)*y)\approx1$
		\end{enumerate}
		instead of {\rm(S2)} then it is called a {\em strong skew Hilbert algebra}. If $(S,\leq)$ is a lattice and $\mathbf S$ satisfies conditions {\rm(S1)}, {\rm(S2')}, {\rm(S3)} and {\rm(S4)} then $\mathbf S$ is called a {\em lattice skew Hilbert algebra}.
	\end{definition}
	
	It is worth noticing that every lattice skew Hilbert algebra is strong.
	
	From (S4) we obtain $L(1*x)=L(U(1,x),1*x)=L(x)$, i.e.
	\begin{enumerate}
		\item[(1)] $1*x\approx x$
	\end{enumerate}
	and $x\in L(x)=L(U(y,x),y*x)$ whence $x\leq y*x$, i.e.
	\begin{enumerate}
		\item[(2)] $x*(y*x)\approx1$.
	\end{enumerate}
	From (S1) and (2) we conclude $x*1\approx x*(x*x)\approx1$, i.e.
	\begin{enumerate}
		\item[(3)] $x*1\approx1$
	\end{enumerate}
	and hence $1$ is the top element of the poset $(S,\leq)$.
	
	\begin{example}\label{rem1}
		Every poset with top element can be converted into a strong skew Hilbert algebra. Namely, if $(S,\leq,1)$ is a poset with top element $1$ and $*$ denotes the binary operation on $S$ defined by
		\[
		x*y:=\left\{
		\begin{array}{ll}
			1, & \text{if }x\leq y; \\
			y, & \text{otherwise,}
		\end{array}
		\right.
		\]
		then $(S,\leq,*,1)$ is a strong skew Hilbert algebra.
	\end{example}
	
	\begin{remark}\label{rem2}
		If $(S,\leq,*,1)$ is a skew Hilbert algebra then according to {\rm(S4)}
		\[
		L(U(x,y),x*y)=L(y)
		\]
		which shows that there exists the infimum $U(x,y)\land(x*y)$ and hence the previous is equivalent to the equality
		\[
		U(x,y)\land(x*y)=y.
		\]
		Thus, in case $x\geq y$ we obtain $x\land(x*y)=y$.
	\end{remark}
	
	Concerning the relationship between skew Hilbert algebras, Hilbert algebras and sectionally pseudocomplemented posets, the following can be said.
	
	\begin{remark}\label{rempsecseudo}
		\
		\begin{itemize}
			\item A skew Hilbert algebra is a Hilbert algebra if and only if it satisfies {\rm(H5)}.
			\item A {\rm(}strong{\rm)} skew Hilbert algebra is a {\rm(}strongly{\rm)} sectionally pseudocomplemented poset if and only if it satisfies the condition
			\[
			L(U(x,y),z)=L(y)\text{ implies }z*(x*y)=1.
			\]
			In other words, a skew Hilbert algebra is a sectionally pseudocomplemented poset if and only if $x*y$ is the sectional pseudocomplement of $x$ with respect to $y$.
		\end{itemize}
	\end{remark}
	
	Comparing Definition~\ref{def3} with Propositions~\ref{prop2} and \ref{prop3}, we conclude immediately that every (stron\-gly) sectionally pseudocomplemented poset can be considered as a (strong) skew Hilbert algebra. The following example shows in part (b) that the converse assertion need not be true.
	
	Further, it is a natural question if every skew Hilbert algebra whose underlying poset is a lattice is strong, i.e.\ if it satisfies identity (S2') automatically. The following example shows in part (a) that this is not the case.
	
	\begin{example}\label{ex1}
		\
		\begin{enumerate}[{\rm(a)}]
			\item If $\mathbf L=(L,\lor,\land)$ denotes the lattice visualized in Fig.~1
			
			\vspace*{-2mm}
			
			\begin{center}
				\setlength{\unitlength}{7mm}
				\begin{picture}(8,8)
					\put(3,1){\circle*{.3}}
					\put(1,3){\circle*{.3}}
					\put(5,3){\circle*{.3}}
					\put(3,5){\circle*{.3}}
					\put(5,5){\circle*{.3}}
					\put(7,5){\circle*{.3}}
					\put(5,7){\circle*{.3}}
					\put(3,1){\line(-1,1)2}
					\put(3,1){\line(1,1)4}
					\put(5,7){\line(-1,-1)4}
					\put(5,7){\line(1,-1)2}
					\put(5,3){\line(-1,1)2}
					\put(5,3){\line(0,1)4}
					\put(2.85,.25){$0$}
					\put(.3,2.85){$a$}
					\put(5.4,2.85){$b$}
					\put(2.3,4.85){$c$}
					\put(5.4,4.85){$d$}
					\put(7.4,4.85){$e$}
					\put(4.85,7.4){$1$}
					\put(3.2,-.75){{\rm Fig.~1}}
				\end{picture}
			\end{center}
			
			\vspace*{4mm}
			
			and $*$ the binary operation on $L$ is defined by
			\[
			\begin{array}{c|ccccccc}
				* & 0 & a & b & c & d & e & 1 \\
				\hline
				0 & 1 & 1 & 1 & 1 & 1 & 1 & 1 \\
				a & d & 1 & d & 1 & d & e & 1 \\
				b & a & a & 1 & 1 & 1 & 1 & 1 \\
				c & 0 & a & b & 1 & d & e & 1 \\
				d & a & a & e & c & 1 & e & 1 \\
				e & 0 & a & d & c & d & 1 & 1 \\
				1 & 0 & a & b & c & d & e & 1
			\end{array}
			\]
			then $(L,\leq,*,1)$ is a skew Hilbert algebra which is not strong since
			\[
			a\not\leq e=d*b=(a*b)*b.
			\]
			$(L,\leq,*,1)$ is not a Hilbert algebra since
			\[
			a*(0*e)=a*1=1\not\leq e=d*e=(a*0)*(a*e).
			\]
			$(L,\lor,\land)$ is not a sectionally pseudocomplemented lattice since the sectional pseudocomplement of $a$ with respect to $b$ does not exist.
			\item Let $\mathbf P=(P,\leq,1)$ denote the poset with top element $1$ visualized in Fig.~2
			
			\vspace*{-2mm}
			
			\begin{center}
				\setlength{\unitlength}{7mm}
				\begin{picture}(6,6)
					\put(3,1){\circle*{.3}}
					\put(5,1){\circle*{.3}}
					\put(1,3){\circle*{.3}}
					\put(3,3){\circle*{.3}}
					\put(5,3){\circle*{.3}}
					\put(3,5){\circle*{.3}}
					\put(3,1){\line(-1,1)2}
					\put(3,1){\line(0,1)4}
					\put(3,1){\line(1,1)2}
					\put(5,3){\line(-1,1)2}
					\put(5,3){\line(0,-1)2}
					\put(3,5){\line(-1,-1)2}
					\put(2.85,.25){$a$}
					\put(4.85,.25){$b$}
					\put(.3,2.85){$c$}
					\put(3.4,2.85){$d$}
					\put(5.4,2.85){$e$}
					\put(2.85,5.4){$1$}
					\put(2.2,-.75){{\rm Fig.~2}}
				\end{picture}
			\end{center}
			
			\vspace*{4mm}
			
			and $*$ the binary operation on $P$ defined by
			\[
			\begin{array}{c|cccccc}
				* & a & b & c & d & e & 1 \\
				\hline
				a & 1 & b & 1 & 1 & 1 & 1 \\
				b & c & 1 & c & d & 1 & 1 \\
				c & a & b & 1 & d & e & 1 \\
				d & a & b & c & 1 & e & 1 \\
				e & a & b & c & d & 1 & 1 \\
				1 & a & b & c & d & e & 1
			\end{array}
			\]
			Then $(P,\leq,*,1)$ is a skew Hilbert algebra. Clearly, it is not a lattice. And it is not a sectionally pseudocomplemented poset since there is no sectional pseudocomplement of $c$ with respect to $a$. Moreover, $(P,\leq,*,1)$ is not a strong skew Hilbert algebra since
			\[
			b\not\leq a=c*a=(b*a)*a.
			\]
		\end{enumerate}
	\end{example}
	
	We close this section by showing that lattice skew Hilbert algebras form a variety.
	
	\begin{theorem}
		Let $\mathbf L=(L,\lor,\land,*,1)$ be a lattice with a binary operation $*$ and a constant $1$. Then $\mathbf L$ is a lattice skew Hilbert algebra if and only if it satisfies the following identities:
		\begin{enumerate}[{\rm(L1)}]
			\item $x*(x\lor y)\approx1$,
			\item $x*((x*y)*y)\approx1$,
			\item $((x\lor y)*z)*(x*z)\approx1$,
			\item $(x\lor y)\land(x*y)\approx y$.
		\end{enumerate}
	\end{theorem}
	
	\begin{proof}
		Identity (L4) is equivalent to (S4). If $x\leq y$ and (L1) holds then $x*y=x*(x\lor y)=1$. If, conversely, $x*y=1$ and (L4) holds then $x\leq(x\lor y)\land(x*y)=y$. This shows that (L1) and (L4) are equivalent to (S1) and (S4). Identity (L2) coincides with (S2'), and (L3) is equivalent to (S3).
	\end{proof}
	
	Analogously to the case of skew Hilbert algebras from (L4) we obtain
	\[
	1*x\approx(1\lor x)\land(1*x)\approx x,
	\]
	i.e.\ (1), and
	\[
	x=(y\lor x)\land(y*x)\leq y*x,
	\]
	i.e.\ (2), and from (L1) and (2) we conclude
	\[
	x*1\approx x*(x*(x\lor x))\approx x*(x*x)\approx1,
	\]
	i.e.\ (3), and hence $1$ is the top element of the lattice $(L,\lor,\land)$.
	
	\begin{corollary}
		The class $\mathcal V$ of lattice skew Hilbert algebras forms a variety determined by the identities for lattices and identities {\rm(L1)} -- {\rm(L4)}.
	\end{corollary}
	
	It is evident that the variety of lattice Hilbert algebras and the variety of sectionally pseudocomplemented lattices (see Proposition~\ref{prop1}) are subvarieties of $\mathcal V$. Precisely, a lattice skew Hilbert algebra is a sectionally pseudocomplemented lattice if and only if it satisfies the identity $z\lor y\leq x*((x\lor y)\land(z\lor y))$. It is worth noticing that this identity is not satisfied by the lattice skew Hilbert algebra from Example~\ref{ex1} since
	\[
	e\lor0=e\not\leq b=c*b=c*(c\land e)=c*((c\lor0)\land(e\lor0)).
	\]
	
	\begin{remark}
		In contrast to Examples~\ref{ex1}, sectionally pseudocomplemented lattices satisfy the identity $x\leq(x*y)*y$. This can be seen as follows: If $(P,\lor,\land,*)$ is a sectionally pseudocomplemented lattice and $a,b\in P$ then
		\[
		((a*b)\lor b)\land(a\lor b)=(a*b)\land(a\lor b)=(a\lor b)\land(a*b)=b
		\]
		and hence
		\[
		a\leq a\lor b\leq(a*b)*b.
		\]
	\end{remark}
	
	\section{Applications}\label{sec:app}
	The aim of this section is providing several motivating examples of skew Hilbert algebras. In particular, we highlight the connections between skew Hilbert algebras and orthomodular implication algebras \cite{CL01}. Precisely, we will show that orthomodular implication algebras are indeed (term equivalent to) a subvariety of skew Hilbert algebras, axiomatized by two further identities. Then, making use of this result, we axiomatize the class of generalized orthomodular lattices within the class of skew Hilbert algebras, introduced by Janowitz in \cite{Ja68} (see also \cite{Be85}), and we show that they form in fact a variety. As a consequence, we obtain a characterization of orthomodular lattices in the framework of skew Hilbert algebras. Subsequently, we show that (strong) skew Hilbert algebras can be regarded as proper generalizations of lattices with sectional antitone involutions (i.e.\ basic algebras, see \cite{CHK01}) representing a common generalization of orthmodular lattices and MV-algebras.
	
	%The next propositions provide remarkable examples of skew Hilbert algebras.
	%\begin{problem} Characterize orthomodular lattices among skew Hilbert algebras.
	%\end{problem}
	
	In \cite{CL01}, two of the present authors together with R.~Hala\v s introduced the concept of an orthomodular implication algebra. The main motivation for discussing this notion was generalizing to the case of orthomodular lattices the fact that in Boolean algebras the properties of the implication operation can be modeled by a so-called implication algebra. This structure itself can be considered as a join-semilattice with $1$, whose principal filters are Boolean algebras.
	As discussed in \cite{CL01}, orthomodular implication algebras can be axiomatized as follows:
	
	\begin{definition}\label{df:oia}
		An \emph{orthomodular implication algebra} is an algebra $\mathbf A=(A, \cdot, 1)$, of type $(2,0)$ such that the following conditions hold:
		\begin{enumerate}[{\rm(O1)}]
			\item $xx=1$,
			\item $x(yx)=1$,
			\item $(xy)x=x$,
			\item $(xy)y=(yx)x$,
			\item $(((xy)y)z)(xz) = 1$,
			\item $(((((((((xy)y)z)z)z)x)x)z)x)x = (((xy)y)z)z$.
		\end{enumerate}
	\end{definition}
	For the reader's convenience, let us recall that, setting $x\lor y = (xy)y$, for any orthomodular implication algebra $\mathbf A$, $(A,\lor,1)$ is a join-semilattice with top element $1$, whose induced order is specified by $x\leq y  \text{ if and only if } xy = 1$.
%	\begin{align*}\label{cond:ordr}
%		\begin{split}
%			x\leq y & \text{ if and only if } xy = 1, \\
%			& \ \ x\lor y = (xy)y.
%		\end{split}
%		\tag{A}
%	\end{align*}
	\begin{definition}{\rm\cite[ Definition 5]{CL01}} An \emph{orthomodular join-semilattice} is an algebra of the form $\mathbf{A}=(A,\lor,1,({}^{p}:p\in A))$ where $(A, \lor, 1)$ is a join-semilattice with greatest element 1 and for each $p\in A$, ${}^p$ is a unary operation on $[p,1]$ such that $([p,1],\lor,
		\land_{p},{}^p, p, 1)$ is an orthomodular lattice where $\land_{p}$ denotes the meet-operation corresponding to the partial order induced by $\lor$.
	\end{definition}
	\begin{lemma}\label{lem: ortjoinsem<->ortimplal}{\rm\cite[Theorems 2 and 3]{CL01}} Let $\mathbf{A}=(A,\cdot,1)$ and $\mathbf{B}=(B,\lor,1,({}^p:p\in B))$ be an orthomodular implication algebra and an orthomodular join-semilattice, respectively. Setting, for any $x,y\in A$ and $z\in[x,1]$ \[x\lor y:=(xy)y \text{ and } z^x:=zx,\] and, for any $x,y\in B$, \[ x\cdot y:=(x\lor y)^y,\] the following hold:
		\begin{enumerate}[{\rm(i)}]
			\item  $\mathcal{S}(\mathbf{A})=(A,\lor,1,({}^p:p\in A))$ is an orthomodular join-semilattice.
			\item $\mathcal{A}(\mathbf{B})=(B,\cdot,1)$ is an orthomodular implication algebra.
		\end{enumerate}
	\end{lemma}
	%\begin{lemma}\cite[Theorem 4]{CL01} For fixed base set $A$ the mappings $\mathcal{S}$ and $\mathcal{A}$ are mutually inverse bijections between the set of all orthomodular implication algebras over $A$ and the set of all orthomodular join-semilattices over $A$.
	%\end{lemma}
	Interestingly enough, any orthomodular implication algebra naturally gives rise to a skew Hilbert algebra, as the following lemma shows.
	
	\begin{theorem}\label{lem:oia-sha}
		Let $\mathbf A=(A, \cdot, 1)$ be an orthomodular implication algebra. Then, upon setting $x* y= (x\lor y)y$, and $x\leq y$ if and only if $xy=1$, then $\mathbf A=(A,\leq, *,1)$ is a skew Hilbert algebra, that satisfies 
		\begin{enumerate}[{\rm(i)}]
			\item $(x* y)* y=(y* x)* x$,
			\item $(((((((((x* y)* y)* z)* z)* z)* x)* x)* z)* x)* x = (((x* y)* y)* z)* z$.
		\end{enumerate}
		Conversely, if $\mathbf{S}=(S,\leq,*,1)$ is a skew Hilbert algebra satisfying {\rm(i)} and {\rm(ii)}, then setting $x\cdot y=(x\lor y)* y$, where $x\lor y:=(x* y)* y$, one has that $(S,\cdot,1)$ is an orthomodular implication algebra.
	\end{theorem}
	\begin{proof}
		(S1) If $a\leq b$, then $a* b=(a\lor b)b=bb=1$. Conversely, if $(a\lor b) b=a* b=1$, then, from condition \eqref{cond:ordr}, $a\lor b\leq b$, i.e.\ $a\leq b$. (S2) Suppose that $b* a=1$. Hence $b\leq a$. Consider $a*((a* b)* b)$. Then,
		\begin{align*}
			a*((a* b)* b) & =  (a\lor ((((a\lor b)b)\lor b )b))((((a\lor b)b)\lor b )b)=\\
			& = (a\lor (((ab)\lor b )b))(((a b)\lor b )b)=\\
			& = (a\lor ((ab)b))((a b)b)=\\
			& = (a\lor (a\lor b))(a\lor b)=\\
			& = (a \lor a) a=\\
			& = aa=1.
		\end{align*}
		(S3) Suppose that $a\leq b$. Now, $(b* c)* (a* c)=((b\lor c)c)* ((a\lor c)c)$. Since $a\leq b$, $a\lor c\leq b\lor c$. Then by \cite[Theorem (ix)]{CL01}, $(b\lor c)c\leq (a\lor c)c$. Consequently, $(b* c)* (a* c)=(((b\lor c)c)\lor ((a\lor c)c))((a\lor c)c)=((a\lor c)c)((a\lor c)c)=1$.\\
		(S4) $L(U(a,b)a* b)=L(U(a,b),(a\lor b )b)= L(U(a,b),(a\lor b)^b)$, in the interval $[b, 1]$, which is an orthomodular lattice (see the proof of \cite[Theorem 4]{CL01}). Therefore, $L(U(a,b),(a\lor b)^b)=L(a\lor b,(a\lor b)^b)=L(b)$, since $(a\lor b)\land (a\lor b)^b=b$.
		Conditions (i) and (ii) are immediate.\\
		%-------------------------------------------------
		%-------------------------------------------------
		Concerning the converse direction, we prove that, setting, for any $a,b\in S$, $a\lor b:= (a* b)* b$ and $a^{p}:=a* p$, for any $p\in S$ and $a\in [p,1]$, $(S,\lor,1,({}^{p}:p\in S))$ is an orthomodular join-semilattice. Therefore, by applying Lemma \ref{lem: ortjoinsem<->ortimplal}(ii), one has that, setting $x\cdot y=(x\lor y)* y$, $(S,\cdot, 1)$ is an orthomodular implication algebra.\\
		Let $a,b\in S$. Clearly, $a,b\leq (a* b)* b=(b* a)* a$, by (i). Now, suppose that $a,b\leq c$. Then, by applying (S3) twice, one has $(a* b)* b\leq (c* b)* b=(b* c)* c=1* c=c$. We conclude that $(S,\lor,1)$ is a join-semilattice with $1$ as its top element. Let $p\in S$. Clearly, the operation ${}^{p}$ on the interval $[p,1]$, is an antitone involution by (S3) and $(a* p)* p= a\lor p = a$, for any $a\in [p,1]$. Moreover, setting, for any $x,y\in [p,1]$, $x\land_{p}y:= (x^p \lor y^p )^p$, it is easily seen that $\land_{p}$ is the meet operation whose dual is $\lor$ on $[p,1]$. Since, for any $x\in [p,1]$ one has that $x\land x^{p}=x\land (x* p) =p$ (by (S4)), one has that $x\land_{p}(x* p)=p$ and so $x\lor x^p =(x^p \land_p x)^p =p^p = 1$. We conclude that $([p,1],\lor,\land_{p},{}^{p},p,1)$ is an ortholattice. Finally, assume that $p\leq x\leq y$. By (ii), one has:
		\begin{align*}
			y & = (x\lor y)\lor p=\\
			& = (((x* y)* y)* p)* p=\\
			& = (((((((((x* y)* y)* p)* p)* p)* x)* x)* p)* x)* x=\\
			& = (((((x\lor y)\lor p)* p)\lor x)* p)\lor x=\\
			& = (y^p \lor x)^p \lor x=\\
			& = (y\land_p x^p ) \lor x.
		\end{align*}
		Therefore, $([p,1],\lor,\land_{p},{}^{p},p,1)$ is an orthomodular lattice.
	\end{proof}
	We note that any orthomodular implication algebra induces a strong skew Hilbert algebra. However, this algebra, in general, may not be lattice-ordered, since the underlying poset could be a join-semilattice only.
	\begin{definition}\label{def:secortlat}A \emph{sectional orthomodular lattice} is a structure $\mathbf{A}=(A,\lor,\land,0,({}^{p}:p\in A))$ such that $(A,\lor,\land,0)$ is a lattice with a bottom  element $0$ and, for any $p\in A$, ${}^{p}:[0,p]\rightarrow[0,p]$ is an antitone involution on $([0,p],\leq)$ such that $([0,p],\lor,\land,{}^{p},0,p)$ is an orthomodular lattice.
	\end{definition}
	%Note that, for any sectional orthomodular lattice $\mathbf{A}$ and any $p\in A$, ${}^{p}$ is a \emph{partial} operation over $A$. In fact, $p$ may not be defined over the $A$, in general. However, ${}^p$ can be extended to a total operation over $A$, that we denote again by ${}^p$, by setting, for any $x\in A$, $x^p$ to be $(x\land p)^p$. Note that these new operations are still antitone although, in general, $x^{pp}\approx x$ need not hold. However, we have that $x^{pp}=x\land p\leq x$, for any $p,x\in A$.
	
	For the reader's convenience, let us recall the notion of generalized orthomodular lattice, which will play a relevant role in the development of the present section.
	\begin{definition}\label{def: gnrlzdrthmdlrlttc}{\rm\cite{Ja68}} A \emph{generalized orthomodular lattice} is a sectional orthomodular lattice $\mathbf{A}=(A,\lor,\land,0,({}^{p}:p\in A))$ satisfying, for any $x,y,p\in A$, the following additional condition: 
		\[x\leq y\leq p \text{ entails } x^y = x^p \land y\]
	\end{definition}
	From now on, we will denote by $\mathcal{GOML}$, the class of generalized orthomodular lattices.
	\begin{lemma}\label{lem:gom-eq} Let $\mathbf{A}=(A,\lor,\land,0,({}^{p}:p\in A))$ be a sectional orthomodular lattice. Then $\mathbf{A}$ is generalized orthomodular if and only if it satisfies:
		\[(x\land a)^{a}\approx (x\land a)^{a\lor b}\land a.\]
	\end{lemma}
	\begin{proof}
		Note that, since $x\land a\leq a\leq a\lor b$, from the above condition one has $(x\land a)^a =(x\land a)^{a\lor b}\land a$. The converse direction is trivial.
	\end{proof}
	Given a lattice $\mathbf{A}$, let us denote by $\mathbf{A}^{\partial}=(A,\lor^{\partial},\land^{\partial})$ the dual of $\mathbf{A}$, i.e.\ the lattice obtained from $\mathbf{A}$ by setting, for any $x,y\in A$, $x\leq^{\mathbf{A}^\partial}y$ if $y\leq^{\mathbf{A}}x$. Clearly, if $\mathbf{A}$ is an orthomodular lattice, then its lattice dual $\mathbf{A}^\partial$ equipped with an antitone involution defined in the obvious way, is again an orthomodular lattice.
	\begin{remark}\label{rem: secorth->joinorthlat} Let $\mathbf{A}=(A,\lor,\land,0,({}^{p}:p\in A))$ be a sectional orthomodular lattice. It is easily seen that, once endowed with unary operations inherited by $\mathbf{A}$, $A^{\partial}$ is an orthomodular join-semilattice {\rm(}with $0$ as its greatest element{\rm)} which is also a lattice.
	\end{remark}
	Lemma \ref{lem: genorthlat->orthoinsemi} shows that it is possible to frame by means of two natural identities the theory of generalized orthomodular lattices within the class of orthomodular join-semilattices.
	\begin{lemma}\label{lem: genorthlat->orthoinsemi}Let $\mathbf{A}=(A,\lor,\land,0,({}^{p}:p\in A))$ be a generalized orthomodular lattice. Then $\mathbf{A}^\partial=(A,\lor^{\partial},\land^{\partial},0,({}^{p}:p\in A))$ is a {\rm(}lattice-ordered{\rm)} orthomodular join-semilattice satisfying, for any $x,y,z\in A$:
		\begin{equation}
			(x\lor y)^y=(x\lor y)^{y\land z}\lor y.\tag{B}\label{A}    
		\end{equation}
		
		Conversely, for any lattice-ordered orthomodular join-semilattice $\mathbf{A}=(A,\lor,\land,1,({}^p:p\in A))$ satisfying \eqref{A}, $\mathbf{A}^\partial=(A,\lor^{\partial},\land^{\partial},1,({}^p:p\in A))$ is a generalized orthomodular lattice. 
	\end{lemma}
	\begin{proof} Clearly $(A, \lor ^\partial, 0)$ is a join-semilattice. Moreover, from Definition \ref{def: gnrlzdrthmdlrlttc} and Lemma \ref{lem:gom-eq} follows that $\mathbf{A}^\partial=(A,\lor^{\partial},\land^{\partial},0,({}^{p}:p\in A))$ is an orthomodular join-semilattice that satisfies equation \eqref{A}. The converse is immediate.
	\end{proof}
	
	Making use of Lemma \ref{lem: genorthlat->orthoinsemi} we obtain the following theorem.
	\begin{theorem} Let $\mathbf{A}=(A,\lor,\land,0,({}^p:p\in A))$ be a generalized orthomodular lattice. Then setting, for any $x,y\in A$, $xy = (x\lor^\partial y)^y$, $\mathbf{A}^\partial=(A,\cdot,0)$ is a lattice-ordered orthomodular implication algebra satisfying
		\[xy\approx ((x\lor y)(y\land z))\lor y.\tag{$\text{B}^*$}\label{A1}\]
		Conversely, any lattice-ordered orthomodular implication algebra satisfying \eqref{A1} induces a generalized orthomodular lattice.
	\end{theorem}
	\begin{proof}
		By Lemma \ref{lem: genorthlat->orthoinsemi}, $\mathbf{A}^\partial=(A,\lor^{\partial},\land^{\partial},0,({}^{p}:p\in A))$ is a lattice-ordered orthomodular join-semilattice satisfying equation \eqref{A}. Hence, by Lemma \ref{lem: ortjoinsem<->ortimplal}, $\mathcal{A}(\mathbf{A}^{\partial})= (A,\cdot,0)$ is an orthomodular implication algebra satisfying \eqref{A1}. Conversely, if $\mathbf{A}=(A,\cdot,1)$ is a lattice-ordered orthomodular implication algebra satisfying \eqref{A1}, then, by Lemma \ref{lem: ortjoinsem<->ortimplal}, $\mathcal{S}(\mathbf{A})=(A,\lor,\land, 1,({}^p:p\in A))$ is a lattice-ordered orthomodular join-semilattice satisfying \eqref{A}, and so $\mathcal{S}(\mathbf{A})^\partial$ is a generalized orthomodular lattice, by Lemma \ref{lem: genorthlat->orthoinsemi}.
	\end{proof}
	The following corollary is a direct consequence of Theorem \ref{lem:oia-sha}.
	\begin{corollary}\label{cor: goml->lskewhilbalg} $\mathcal{GOML}$ is term equivalent to the variety of lattice skew Hilbert algebras satisfying the following identities:
		\begin{enumerate}[{\rm[i]}]
			\item $(x* y)* y\approx (y* x)* x$,
			\item $(((((((((x* y)* y)* z)* z)* z)* x)* x)* z)* x)* x \approx (((x* y)* y)* z)* z$,
			\item $x* y\approx ((x\lor y)(y\land z))\lor y$.
		\end{enumerate}
	\end{corollary}
	
	It is well known that orthomodular lattices are generalized orthomodular lattices with an upper bound $1$. Therefore, the above results characterize orthomodular lattices in the variety of lattice skew Hilbert algebras with bottom element $0$. 
	\begin{corollary}\label{lem: ortmodlrskewhilbalg}
$\mathcal{GOML}$ forms a variety.	The variety $\mathcal{OML}$ of orthomodular lattices is term equivalent to the variety of lattice skew Hilbert algebras with bottom element $0$ satisfying conditions {\rm(i)} -- {\rm(iii)} of Corollary \ref{cor: goml->lskewhilbalg}.
	\end{corollary}
	
	%\begin{proof} Concerning $(S1)$, just note that $a\leq b$ entails that $a* b=(a\lor b)'\lor b = b\lor b'=1$. Conversely, if $a* b = 1$, then by $b\leq a\lor b$ and the orthomodular law, we have $a\lor b= b$, i.e.\ $a\leq b$. As regards $(S2')$, just note that, for any $x,y\in A$, $(x* y)* y = x\lor y$ (see \cite{Be85}). $(S3)$ can be easily proven upon noticing that $a\leq b$ entails that $a\lor c\leq b\lor c$ and, since ${}'$ is antitone, $(b\lor c)'\leq(a\lor c)'$ and $b* c=(b\lor c)'\lor c\leq(a\lor c)'\lor c=a* c$. Finally, we have that $L(U(a,b),a* b)=L(a\lor b, a* b)=L(a\lor b, (a\lor b)'\lor b)=L((a\lor b)\land((a\lor b)'\lor b))=L(b)$, by the orthomodular law.
	%\end{proof}
	In what follows, given an orthomodular lattice $\mathbf{A}$, we will denote by $\mathcal{H}(\mathbf{A})$ its associated skew Hilbert algebra.\\
	Note that, in general, if $\mathbf{A}$ is an orthomodular lattice, then $\mathcal{H}(\mathbf{A})$ need not be a relatively pseudocomplemented lattice as the next example shows.
	\begin{example} Consider the orthomodular lattice $\mathbf{MO}_2$ depicted in Fig.\ 3 with $*$ defined as follows:
		\[
		\begin{array}{c|cccccc}
			*  & 0  & a & a' & b & b' & 1 \\
			\hline
			0  & 1  & 1 & 1  & 1 & 1  & 1 \\
			a  & a' & 1 & a' & b & b' & 1 \\
			a' & a  & a & 1  & b & b' & 1 \\
			b  & b' & a & a' & 1 & b' & 1 \\
			b' & b  & a & a' & b & 1  & 1 \\
			1  & 0  & a & a' & b & b' & 1
		\end{array}
		\]
		Then $({\rm MO}_2,\lor,\land,*,0,1)$ is a lattice skew Hilbert algebra. However, let us observe that it is not pseudocomplemented. Indeed, for $x=a$, $y=0$, and $z=b$, one has
		\[b\lor 0\not\leq a' = a* 0 = a* (a\land b) = a*((a\lor 0)\land(b\lor 0)),\]
		i.e.\ the first condition of Proposition \ref{prop1} fails.
		\vspace*{-2mm}
		
		\begin{center}
			\setlength{\unitlength}{7mm}
			\begin{picture}(8,6)
				\put(4,1){\circle*{.3}}
				\put(1,3){\circle*{.3}}
				\put(3,3){\circle*{.3}}
				\put(5,3){\circle*{.3}}
				\put(7,3){\circle*{.3}}
				\put(4,5){\circle*{.3}}
				\put(4,1){\line(-3,2)3}
				\put(4,1){\line(-1,2)1}
				\put(4,1){\line(1,2)1}
				\put(4,1){\line(3,2)3}
				\put(4,5){\line(-3,-2)3}
				\put(4,5){\line(-1,-2)1}
				\put(4,5){\line(1,-2)1}
				\put(4,5){\line(3,-2)3}
				\put(3.85,.25){$0$}
				\put(.3,2.85){$a$}
				\put(2.3,2.85){$b$}
				\put(5.4,2.85){$b'$}
				\put(7.4,2.85){$a'$}
				\put(3.3,5.4){$1=0'$}
				\put(3.2,-.75){{\rm Fig.\ 3}}
			\end{picture}
		\end{center}
		
		\vspace*{4mm}
	\end{example}
	The next theorem shows that, indeed, orthomodular lattices inducing sectionally pseudocomplemented lattices are Boolean.  
	\begin{theorem} Let $\mathbf{A}=(A,\lor,\land,'\,0,1)$ be an orthomodular lattice. Then $\mathcal{H}(\mathbf{A})$ is sectionally pseudocomplemented if and only if $\mathbf{A}$ is a Boolean algebra.
	\end{theorem}
	\begin{proof}
		Clearly any Boolean algebra is pseudocomplemented by setting $x* y=x'\lor y$. Concerning the converse direction, by Remark \ref{rempsecseudo}, for any $a,b\in A$, $a* b$ is the sectional pseudocomplement of $a$ with respect to $b$. Therefore, $x* 0 = (x \lor 0)'\lor 0 = x'$ is the largest element $c\in A$ such that $x\land c = 0$. Hence, we have that the following condition is fulfilled:
		\[x\land y= 0\quad\text{if and only if}\quad y\leq x'.\]
		In other words, $\mathbf{A}$ is uniquely complemented, i.e.\ $\mathbf A$ is in fact a Boolean algebra. 
	\end{proof}
	
	As it has been pointed out above, orthomodular lattices induce prominent examples of (lattice) strong skew Hilbert algebras by setting $x* y:=(x\lor y)'\lor y$. However, it is conceivable to wonder whether any orthomodular lattice can be endowed with a $*$ operation satisfying certain different, preferable conditions. Indeed, given a bounded poset $\mathbf P=(P,\leq,{}',0,1)$ with a unary operation $'$, it seems reasonable to define 
	\begin{equation}\label{eq:pst}
		x*y:=\left\{
		\begin{array}{ll}
			1  & \text{if }x\leq y, \\
			x' & \text{if }y=0, \\
			y  & \text{otherwise,}
		\end{array}
		\right.
	\end{equation}
	for all $x,y\in P$, and then check whether $\mathbb S(\mathbf P)=(P,\leq,*,1)$ is a skew Hilbert algebra.  The following result provides a smooth characterization of bounded posets with a unary operation which lend themselves to accommodate the construction in condition \eqref{eq:pst}.

	\begin{theorem}\label{th8}
		Let $\mathbf P=(P,\leq,{}',0,1)$ be a bounded poset with a unary operation $'$. Then the following are equivalent:
		\begin{enumerate}[{\rm(i)}]
			\item $\mathbb S(\mathbf P)$ is a skew Hilbert algebra,
			\item $\mathbb S(\mathbf P)$ is a strong skew Hilbert algebra,
			\item for all $x\in P$ the following hold:
			\begin{enumerate}[{\rm(a)}]
				\item $x'=1$ if and only if $x=0$,
				\item $'$ is antitone,
				\item $x\leq x''$,
				\item $L(x,x')=\{0\}$.
			\end{enumerate}
		\end{enumerate}
	\end{theorem}
	
	\begin{proof}
		Let $\mathbb S(\mathbf P)=(P,\leq,*,1)$ and $a,b,c\in P$. By definition of $*$ we have that $\mathbb S(\mathbf P)$ satisfies the identity $1*x\approx x$. First we prove that (i) and (iii) are equivalent.
		\begin{enumerate}[(S1)]
			\item It is easy to see $\mathbb S(\mathbf P)$ satisfies (S1) if and only if $\mathbf P$ satisfies (a).
			\item Since we have
			\[
			(a*b)*b=\left\{
			\begin{array}{ll}
				1*b=b\geq a & \text{if }a\leq b, \\
				a''         & \text{if }b=0, \\
				b*b=1\geq a & \text{otherwise,}
			\end{array}
			\right.
			\]
			$\mathbb S(\mathbf P)$ satisfies (S2) if and only if $\mathbf P$ satisfies (c).
			\item Since in case $a\leq b$ we have
			\begin{align*}
				b*c\leq1=a*c & \quad\text{ if }a\leq c, \\
				b*c\leq a*c\text{ is equivalent to }b'\leq a' & \quad\text{ if }c=0, \\
				b*c=c=a*c & \quad\text{ otherwise,}
			\end{align*}
			because in the last case $b\not\leq c\neq0$, $\mathbb S(\mathbf P)$ satisfies (S3) if and only if $\mathbf P$ satisfies (b ).
			\item Since we have
			\[
			L(U(a,b),a*b)=\left\{
			\begin{array}{ll}
				L(b,1)=L(b)      & \text{if }a\leq b, \\
				L(a,a')          & \text{if }b=0, \\
				L(U(a,b),b)=L(b) & \text{otherwise,}
			\end{array}
			\right.
			\]
			$\mathbb S(\mathbf P)$ satisfies (S4) if and only if $\mathbf P$ satisfies (d).
		\end{enumerate}
		From the proof of the equivalence of (S2) and (b) we see that $\mathbb S(\mathbf P)$ already satisfies (S2') whenever $\mathbf P$ satisfies (b).
	\end{proof}

	As it has been pointed out above, orthomodular join-semilattices can be framed within the theory of (strong) skew Hilbert algebras. It is natural to ask whether any skew Hilbert algebra can be regarded as a poset having sectional antitone operations. In the sequel, we provide a positive answer by proving that any skew Hilbert algebra can be regarded as a poset whose sections can be endowed with a Browerian pseudocomplement. As a consequence, we conclude that the class of skew Hilbert algebras can be regarded as a proper generalization of lattices with sectional antitone involutions \cite{CHK01}. 
	\begin{definition}\label{def:bcpo} A poset with \emph{sectional Browerian pseudocomplements} is a structure $\mathbf{A}=(A,\leq,1,({}^p :p\in A))$ such that $(A,\leq,1)$ is a poset with top element $1$, and, for any $p\in A$, $([p,1],\leq,{}^p ,1)$ is a poset with a unary operation ${}^p$ such that, for any $x,y\in[p,1]$:
		\begin{enumerate}[{\rm(BP1)}]
			\item $x\leq y$ implies $y^p \leq x^p$;
			\item $x\leq x^{pp}$;
			\item $L(x, x^p)=L(p)$.
		\end{enumerate}
	\end{definition}
	In the sequel we will denote by $\mathcal{PSB}$, the class of posets with sectional Browerian pseudocomplements. The next remark shows that the above definition makes sense.
	\begin{remark}\label{rem:propbcpo}
		Note that any poset $\mathbf{A}=(A,\leq,1,({}^p :p\in A))$ with sectional Browerian pseudocomplements satisfies, for any $x\in A$: \[1^x\approx x\quad\text{and}\quad x^x \approx 1.\]
		Indeed, let $p\in A$. One has $L(1^p )= L(1,1^p )=L(p)$. Moreover, $1^p = p$ entails $1\leq 1^{pp} = p^p$. We conclude $p^p = 1$. Therefore, if $a\in [p,1]$, then $a\leq 1$ entails $p=1^p\leq a^p$, i.e.\ $[p,1]$ is closed under ${}^p$.
	\end{remark}
	
	\begin{proposition}\label{thm:skh-bcpo} Let $\mathbf{S}=(S,\leq,*,1)$ be a skew Hilbert algebra. Then setting, for any $p\in S$ and $x\in[p,1]$, $x^p:=x* p$, $(S,\leq,1,({}^p :p\in S))$ is a poset with sectional Browerian pseudocomplements.
	\end{proposition}
	\begin{proof}
		Let $\mathbf{S}=(S,\leq,*,1)$ be a skew Hilbert algebra. (BP1) holds by (S3). (BP2) is a consequence of (S2), while (BP3) follows from (S4), since $x\geq p$.
	\end{proof}
	
	Proposition~\ref{thm:skh-bcpo} shows that every skew Hilbert algebra naturally gives rise to a poset with sectional Browerian pseudocomplements. However, if conditions (BP1) and (BP2) are supposed to hold for all elements of a poset $\mathbf A$ with sectional Browerian pseudocomplements, then $\mathbf A$ can be turned into a skew Hilbert algebra. Moreover this correspondence is one-to-one.
	\begin{definition} A poset $\mathbf A$ with sectional Browerian pseudocomplements is said to be \emph{strong} if conditions {\rm(BP1)} and {\rm(BP2)} hold for any $x,y,p\in A$.
	\end{definition}
	We denote by $s\mathcal{PSB}$ the class of strong posets with sectional Browerian pseudocomplements.
	\begin{theorem} Let $\mathbf{S}=(S,\leq,*,1)$ and  $\mathbf{A}=(A,\leq,1,({}^p:p\in A))$ be a strong skew Hilbert algebra and a strong poset with sectional Browerian pseudocomplements, respectively.
		\begin{enumerate}[{\rm(i)}]
			\item Setting, for any $x,y\in S$, $x^y:=x* y$, one has that $$\mathbb{B}(\mathbf{S})=(S,\leq,1,({}^p:p\in A))$$ is a strong poset with sectional Browerian pseudocomplements.
			\item Setting, for any $x,y\in A$, $x* y:=x^{y}$, one has that $$\mathbb{H}(\mathbf{A})=(A,\leq,*,1)$$ is a strong skew Hilbert algebra.
			\item $\mathbb{H}(\mathbb{B}(\mathbf{S}))=\mathbf{S}$, and $\mathbb{B}(\mathbb{H}(\mathbf{A}))=\mathbf{A}$.
		\end{enumerate}
	\end{theorem}
	\begin{proof} (i) By Proposition~\ref{thm:skh-bcpo}, $\mathbb{B}(\mathbf{S})$ is a poset with sectional Browerian pseudocomplements. Moreover, it is easily seen that, by (S3) and (S2'), (BP1) and (BP2) hold for any $p\in A$ and any $x,y\in[p,1]$.\\
		(ii) We prove that $\mathbb{H}(\mathbf A)$ satisfies (S1), (S3), (S4) and (S2'). Concerning (S1), assume that $a,b\in A$ are such that $a\leq b$. One has that $1=b^b\leq a^b$, by Remark \ref{rem:propbcpo}, and so $a* b=1$. Conversely, if $a^b=1$, then $a\leq a^{bb}=1^b =b$, again by Remark \ref{rem:propbcpo}. (S2') and (S3) follow directy from properties of strong posets with sectional Browerian pseudocomplements. Concerning (S4), note that $a^b\leq 1$ entails $b=1^b\leq a^{bb}$. Therefore, $a,b\leq a^{bb}$ implies $a^{bb}\in U(a,b)$, $L(U(a,b))\subseteq L(a^{bb})$, and $L(U(a,b),a^b)=LU(a,b)\cap L(a^b)\subseteq L(a^{bb})\cap L(a^b)=L(a^{bb},a^b)=L(b)$, by (BP3).\\
		(iii) Straightforward.
	\end{proof}
	\begin{corollary} The class of strong skew Hilbert algebras and $s\mathcal{PSB}$ are term equivalent.
	\end{corollary}
	\begin{corollary} The class of lattice skew Hilbert algebras and the class of lattice-ordered $s\mathcal{PSB}$'s are term equivalent.
	\end{corollary}
	Upon recalling that a lattice with sectional antitone involutions is a structure $\mathbf{A}=(A,\lor,\land,({}^p:p\in A),0,1)$ such that $(A,\lor,\land,0,1)$ is a bounded lattice and, for any $p\in A$, $([p,1],\lor,\land,{}^p,p,1)$ is a lattice with antitone involution, the following corollary easily follows.
	\begin{corollary} The class of lattices with sectional antitone involutions is term equivalent to the variety of lattice skew Hilbert algebras with bottom element $0$ satisfying \[(x* y)* y\approx (y* x)* x.\]
	\end{corollary}
	
	\section{Special subsets of skew Hilbert algebras}\label{sec: specele}
	In this section, we will describe basic properties of some special subsets of skew Hilbert algebras. 
	In particular, we will highlight the connections between the set of ``closed'' elements of a bounded skew Hilbert algebra and orthoposets. Then, we will investigate the relationships between the set of dense and weakly dense elements in a skew Hilbert algebra. 
	
	It can be noticed that, in the skew Hilbert algebras from Example~\ref{ex1} and Example~\ref{rem1}, the elements of the form $x*0$ form a Boolean algebra. In what follows, we show that, in general, this is not the case.
	
	For any skew Hilbert algebra $\mathbf S=(S,\leq,*,1)$ with bottom element $0$ put $x':=x*0$ for all $x\in S$ and $S':=\{x'\mid x\in S\}$ and let $\mathbb O(\mathbf S)$ denote the bounded poset $(S',\leq,{}',0,1)$. (Observe that $0=1'\in S'$ and $1=0'\in S'$.)
	
	The next theorem describes the connections between skew Hilbert algebras and orthoposets. Recall the definition of $\mathbb{S}$ from Theorem \ref{th8}.
	
	\begin{theorem}\label{th9}
		\
		\begin{enumerate}[{\rm(i)}]
			\item Let $\mathbf S=(S,\leq,*,1)$ be a skew Hilbert algebra with bottom element $0$. Then
			\begin{enumerate}[{\rm(a)}]
				\item $\mathbb O(\mathbf S)$ is an orthoposet,
				\item $\mathbb S(\mathbb O(\mathbf S))=\mathbf S$ if and only if $S'=S$ and $x*y=y$ for all $x,y\in S$ with $x\not\leq y\neq0$.
			\end{enumerate}
			\item Let $\mathbf P=(P,\leq,{}',0,1)$ be an orthoposet and $\mathbb S(\mathbf P)=(P,\leq,*,1)$. Then
			\begin{enumerate}
				\item[{\rm(c)}] $\mathbb S(\mathbf P)$ is a skew Hilbert algebra with bottom element $0$ satisfying $P'=P$ and $x*y=y$ for all $x,y\in P$ with $x\not\leq y\neq0$ {\rm(}wherefrom we conclude that it is strong{\rm)},
				\item[{\rm(d)}] $\mathbb O(\mathbb S(\mathbf P))=\mathbf P$.
			\end{enumerate}
		\end{enumerate}
	\end{theorem}
	
	\begin{proof}
		\
		\begin{enumerate}[(i)]
			\item Let $\mathbb S(\mathbb O(\mathbf S))=(S*0,\leq,\circ,1)$ and $a,b\in S$.
			\begin{enumerate}[(a)]
				\item Since $0,1\in S'$, $(S',\leq,0,1)$ is a bounded poset. Moreover, $'$ is a unary operation on $S'$. Because of (S3), $'$ is antitone. By (S2) we have $a\leq a''$. From this we conclude $a'\leq a'''$ and by (S3) also $a'''\leq a'$. Together we have $a'''=a'$ showing that $'$ is an involution on $S*0$. Finally, because of (S4), we conclude $L(a,a')=L(U(a,0),a')=L(0)=\{0\}$ and due to De Morgan's laws $U(a,a')=(L(a',a))'=0'=1$ showing that $'$ is a complementation on $(S',\leq,0,1)$.
				\item This follows from
				\[
				a\circ b=\left\{
				\begin{array}{ll}
					1=a*b & \text{if }a\leq b, \\
					a'=a*b & \text{if }b=0.
				\end{array}
				\right.
				\]
			\end{enumerate}
			\item
			\begin{enumerate}
				\item[(c)] This follows from Theorem~\ref{th8}.
				\item[(d)] Let $\mathbb O(\mathbb S(\mathbf P))=(P',\leq,{}^+,0,1)$ and $a\in P$. According to Theorem~\ref{th8}, $\mathbb S(\mathbf P)$ is a (strong) skew Hilbert algebra. Moreover, $P'=P$ and $a^+=a'$.
			\end{enumerate}
		\end{enumerate}
	\end{proof}
%	\commento{\tiny { {\bf Example \ref{ex: benzene}}. Note that, by Remark \ref{rempsecseudo}, the sectional pseudocomplement of $a'$ with respect to $0$ should be $a$. However, here it is $b'$.}}
	\begin{remark}
		Theorem~\ref{th9} shows that the mappings $\mathbb O$ and $\mathbb S$ establish a one-to-one correspondence between the skew Hilbert algebras $(S,\leq,*,1)$ satisfying $S'=S$ and $x*y=y$ for all $x,y\in S$ with $x\not\leq y\neq0$ {\rm(}which are automatically strong{\rm)} on the one hand and orthoposets on the other.
	\end{remark} 
	
	%\begin{example}
	%If ${\rm MO}_2:=\{0,a,b,a',b',1\}$ and $\mathbf{MO_2}=({\rm MO}_2,\lor,\land,{}',0,1)$ denotes the {\rm(}modular{\rm)} ortholattice visualized in Fig.~3
	
	%\vspace*{-2mm}
	
	%\begin{center}
	%\setlength{\unitlength}{7mm}
	%\begin{picture}(8,6)
	%\put(4,1){\circle*{.3}}
	%\put(1,3){\circle*{.3}}
	%\put(3,3){\circle*{.3}}
	%\put(5,3){\circle*{.3}}
	%\put(7,3){\circle*{.3}}
	%\put(4,5){\circle*{.3}}
	%\put(4,1){\line(-3,2)3}
	%\put(4,1){\line(-1,2)1}
	%\put(4,1){\line(1,2)1}
	%\put(4,1){\line(3,2)3}
	%\put(4,5){\line(-3,-2)3}
	%\put(4,5){\line(-1,-2)1}
	%\put(4,5){\line(1,-2)1}
	%\put(4,5){\line(3,-2)3}
	%\put(3.85,.25){$0$}
	%\put(.3,2.85){$a$}
	%\put(2.3,2.85){$b$}
	%\put(5.4,2.85){$b'$}
	%\put(7.4,2.85){$a'$}
	%\put(3.3,5.4){$1=0'$}
	%\put(3.2,-.75){{\rm Fig.~3}}
	%\end{picture}
	%\end{center}
	
	%\vspace*{4mm}
	
	%and $*$ the binary operation on ${\rm MO}_2$ defined by
	%\[
	%\begin{array}{c|cccccc}
	%*  & 0  & a & a' & b & b' & 1 \\
	%\hline
	%0  & 1  & 1 & 1  & 1 & 1  & 1 \\
	%a  & a' & 1 & a' & b & b' & 1 \\
	%a' & a  & a & 1  & b & b' & 1 \\
	%b  & b' & a & a' & 1 & b' & 1 \\
	%b' & b  & a & a' & b & 1  & 1 \\
	%1  & 0  & a & a' & b & b' & 1
	%\end{array}
	%\]
	%then $\mathbb S(\mathbf{MO_2})=({\rm MO}_2,\lor,\land,*,1)$ is a sectionally pseudocomplemented %lattice skew Hilbert algebra and, by Theorem~\ref{th9}, $\mathbb O(\mathbb %S(\mathbf{MO_2}))=\mathbf{MO_2}$.
	%\end{example}
	
	\begin{example}\label{ex: benzene}
		If ${\rm O}_6:=\{0,a,b,a',b',1\}$ and $\mathbf{O_6}=({\rm O}_6,\lor,\land,{}',0,1)$ denotes the {\rm(}non-modular{\rm)} ortholattice visualized in Fig.~4
		
		\vspace*{-2mm}
		
		\begin{center}
			\setlength{\unitlength}{7mm}
			\begin{picture}(4,8)
				\put(2,1){\circle*{.3}}
				\put(1,3){\circle*{.3}}
				\put(3,3){\circle*{.3}}
				\put(1,5){\circle*{.3}}
				\put(3,5){\circle*{.3}}
				\put(2,7){\circle*{.3}}
				\put(2,1){\line(-1,2)1}
				\put(2,1){\line(1,2)1}
				\put(2,7){\line(-1,-2)1}
				\put(2,7){\line(1,-2)1}
				\put(1,3){\line(0,1)2}
				\put(3,3){\line(0,1)2}
				\put(1.85,.25){$0$}
				\put(.3,2.85){$a$}
				\put(3.4,2.85){$b$}
				\put(.3,4.85){$b'$}
				\put(3.4,4.85){$a'$}
				\put(1.3,7.4){$1=0'$}
				\put(1.2,-.75){{\rm Fig.~4}}
			\end{picture}
		\end{center}
		
		\vspace*{4mm}
		
		and $*$ the binary operation on ${\rm O}_6$ defined by
		\[
		\begin{array}{c|cccccc}
			*  & 0  & a & a' & b & b' & 1 \\
			\hline
			0  & 1  & 1 & 1  & 1 & 1  & 1 \\
			a  & a' & 1 & a' & b & 1  & 1 \\
			a' & b' & a & 1  & b & b' & 1 \\
			b  & b' & a & 1  & 1 & b' & 1 \\
			b' & a' & a & a' & b & 1  & 1 \\
			1  & 0  & a & a' & b & b' & 1
		\end{array}
		\]
		then $\mathbb S(\mathbf{O_6})=({\rm O}_6,\lor,\land,*,1)$ is a lattice skew Hilbert algebra and, by Theorem~\ref{th9}, $\mathbb O(\mathbb S(\mathbf{O_6}))=\mathbf{O_6}$.
	\end{example}
	
Recall that a Boolean poset, in the sense of Tkadlec \cite{Tk93}, is an orthoposet ${\mathbf P} = (P,\leq,',0,1)$ such that, for any $x,y\in P$:
\[x\land y= 0\quad\text{if and only if}\quad x\leq y'.\] It can be shown that an orthoposet $\mathbf{P}$ is Boolean if and only if the following LU-identity holds (see \cite{CFL}):
\[U(L(x, y), z) \approx U(L(U(x, z), U(y, z))).\]
%	As a generalization of the notion of a paraorthomodular lattice introduced in \cite{GLP}, we define the following concept.
%	
%	A {\em paraorthomodular poset} is a bounded poset $\mathbf P=(P,\leq,{}',0,1)$ with an antitone involution $'$ satisfying for all $x,y\in P$ the condition
%	\begin{align*}
%		\textrm{(P*)} \ \ x\leq y \ \text{and} \ L(x',y)=\{0\} \ \text{imply} \ x=y.
%	\end{align*}
%	
	\begin{proposition}\label{osispara}
		Let $\mathbf S=(S,\leq,*,1)$ be a sectionally pseudocomplemented skew Hilbert algebra with bottom element $0$. Then the orthoposet $\mathbb O(\mathbf{S})$ is Boolean. 
	\end{proposition}
	
	\begin{proof}
		Let $a,b\in S'$ with $a\land b = 0$ (in $S'$) and consider $L(a,b)$ (in $S$).  If $c\in L(a,b)$ then $c''\in\mathbb{O}(\mathbf{S})$ and $c''\leq a,b$. So $c\leq c''\leq 0=a\land b$. Therefore $L(U(b,0),a)=\{0\}$ and, by Remark~\ref{rempsecseudo}, $a\leq b'$.
	\end{proof}
	
	\begin{lemma}\label{triplet}
		Let $\mathbf S=(S,\leq,*,1)$ be a skew Hilbert algebra with bottom element $0$ and $a\in S$. Then $a=a''\land(a''*a)$.
	\end{lemma}
	
	\begin{proof}
		We have $a''\geq a$ by (S2). From Remark~\ref{rem2} we obtain $a=a''\land(a''*a)$.
	\end{proof}
	
	Let $\mathbf S=(S,\leq,*,1)$ be a skew Hilbert algebra with bottom element $0$ and $a\in S$ and define
	\begin{align*}
		F_a & :=\{x\in S\mid x''=a\}, \\
		D(\mathbf S) & :=\{x\in S\mid x'=0\}, \\
		W(\mathbf S) & :=\{x\in S\mid\text{there exists some }y\in S\text{ with }y''*y=x\}.
	\end{align*}
	The elements of $D(\mathbf S)$ and $W(\mathbf S)$ are called {\em dense} and {\em weakly dense}, respectively. Note that (see \cite{CLP20}), for a  Hilbert algebra  $\mathbf S=(S,\leq,*,1)$ with bottom element $0$
	\begin{align*}
		& S'\cap D(\mathbf S)=\{1\}, \\
		& D(\mathbf S)\text{ is an upper subset of }\mathbf S, \\
		& (D(\mathbf S),\leq,*,1)\text{ is a Hilbert subalgebra of }\mathbf S.
	\end{align*}
	%The last statement follows from the fact that $D(\mathbf S)\text{ is an upper subset of }\mathbf S$, 
	%and hence $a, b\in D(\mathbf S)$ implies $b\leq a*b\in D(\mathbf S)$.
	
	\begin{lemma}
		Let $\mathbf S=(S,\leq,*,1)$ be a skew Hilbert algebra with bottom element $0$. Then
		\begin{enumerate}[{\rm(i)}]
			\item $S'\cap D(\mathbf S)=\{1\}$.
			\item $D(\mathbf S)\text{ is an upper subset of }\mathbf S$.
			\item $(D(\mathbf S),\leq,*,1)\text{ is a skew Hilbert subalgebra of }\mathbf S$.
			\item $D(\mathbf S)\subseteq W(\mathbf S)$.
			\item $S'\cap W(\mathbf S) = \{1\}$.
		\end{enumerate}
	\end{lemma}
	
	\begin{proof}
		Suppose that $\mathbf S=(S,\leq,*,1)$ is a skew Hilbert algebra with bottom element.
		\begin{enumerate}[(i)]
			\item Follows from the inclusion given in (iv) and (v).
			\item Let $a\in D(\mathbf S)$ and $a\leq b$. This yields $a*b=1$, and from (S3) we also obtain $(b*0)*(a*0)=1$. Since $a\in D(\mathbf S)$ then $(b*0)*0=1$ which implies $b*0=0$. Hence we get $b\in D(\mathbf S)$.
			\item This follows from the fact that $D(\mathbf S)\text{ is an upper subset of }\mathbf S$ and hence $a,b\in D(\mathbf S)$ implies $b\leq a*b\in D(\mathbf S)$.
			\item If $a\in D(\mathbf S)$ then $a'=0$ and hence $a=1*a=a''*a\in W(\mathbf S)$, i.e., $D(\mathbf S)\subseteq W(\mathbf S)$.
			\item Let $a\in S'\cap W(\mathbf S)$. Then $a=c''*c$ for some $c\in S$ by definition, and $a''=a$ from the proof of Theorem~\ref{th9}. Correspondingly, we have $(c''*c)''=c''*c$ and also $c\leq c''$. Moreover, since $c\leq c''*c$ we have $c''\leq(c''*c)''=c''*c$. By (S4) we get
			\[
			L(c)=L(U(c'',c),c''*c)=L(U(c''),c''*c)=L(c'',c''*c)=L(c'').
			\]
			Thus $c=c''$ and consequently $a=1$. Hence we obtain $S'\cap W(\mathbf S)=\{1\}$.
		\end{enumerate}
	\end{proof}
	
	\begin{example}
		Consider the lattice $(S,\leq)$ visualized in Fig.~5
		
		\vspace*{-2mm}
		
		\begin{center}
			\setlength{\unitlength}{7mm}
			\begin{picture}(4,10)
				\put(2,1){\circle*{.3}}
				\put(1,3){\circle*{.3}}
				\put(3,3){\circle*{.3}}
				\put(1,5){\circle*{.3}}
				\put(3,5){\circle*{.3}}
				\put(2,7){\circle*{.3}}
				\put(2,9){\circle*{.3}}
				\put(2,1){\line(-1,2)1}
				\put(2,1){\line(1,2)1}
				\put(2,7){\line(-1,-2)1}
				\put(2,7){\line(1,-2)1}
				\put(1,3){\line(0,1)2}
				\put(3,3){\line(0,1)2}
				\put(2,7){\line(0,1)2}
				\put(1.85,.25){$0$}
				\put(.3,2.85){$a$}
				\put(3.4,2.85){$b$}
				\put(.3,4.85){$c$}
				\put(3.4,4.85){$d$}
				\put(2.4,6.85){$e$}
				\put(1.85,9.4){$1$}
				\put(1.2,-.75){{\rm Fig.~5}}
			\end{picture}
		\end{center}
		
		\vspace*{4mm}
		
		which is sectionally pseudocomplemented with the following binary operation:
		\[
		\begin{array}{c|ccccccc}
			* & 0 & a & b & c & d & e & 1  \\
			\hline
			0 & 1 & 1 & 1 & 1 & 1 & 1 & 1 \\
			a & d & 1 & b & 1 & d & 1 & 1 \\
			b & c & a & 1 & c & 1 & 1 & 1 \\
			c & d & a & b & 1 & d & 1 & 1 \\
			d & c & a & b & c & 1 & 1 & 1 \\
			e & 0 & a & b & c & d & 1 & 1 \\
			1 & 0 & a & b & c & d & e & 1
		\end{array}
		\]
		Automatically, $\mathbf S=(S,\leq,*,1)$ is a skew Hilbert algebra with bottom element $0$. We have
		\begin{align*}
			D(\mathbf S) & =\{e,1\}, \\
			W(\mathbf S) & =\{a,b,e,1\}
		\end{align*}
		and hence $D(\mathbf S)\neq W(\mathbf S)$. Moreover, $W(\mathbf S)$ is not an upper set of $\mathbf S$ since $b\in W(\mathbf S)$ and $b\leq d$, but $d\notin W(\mathbf S)$.
	\end{example}
	
	\begin{remark}
		However, in the bounded Hilbert algebra case, the sets of dense and weakly dense elements 
		coincide {\rm(}see {\rm\cite[Lemma 3.7]{CLP20}} for the inclusion $W(\mathbf S) \subseteq D(\mathbf S)${\rm)}.
	\end{remark}

	%\begin{enumerate}[{\bf\mbox{Question} 1.}]
	%\item Is $W(\mathbf S)$ a skew Hilbert subalgebra of $\mathbf S$?
	%\end{enumerate}
	
	\section{Congruences in skew Hilbert algebras}\label{sec: structheory}
	In this last section, we will investigate the structure theory for the variety $\mathcal V$ of lattice skew Hilbert algebras (see e.g.\ \cite{CEL}). In particular, we will show that $\mathcal V$ is arithmetical and weakly regular. Moreover, since any congruence on a lattice skew Hilbert algebra is determined by its $1$-coset, a further task will be characterizing these sets by introducing a suitable notion of \emph{filter} and then proving that the complete lattice of filters on a lattice skew Hilbert algebra $\mathbf{L}$ is isomorphic to the complete lattice of congruences on $\mathbf{L}$. Subsequently, by extending an analogous notion for Hilbert algebras (see e.g.\ \cite{BUS}), we will introduce the concept of a {\em deductive system} on a skew Hilbert algebra. A full characterization thereof will follow. Finally, we introduce a notion of ``order-compatible'' congruence for (strong) skew Hilbert algebras which need not be lattice-ordered. In turn, we show that, also in this case, many of the aforementioned results hold.  
	
	First, let us recall the following concepts.
	
	Let $\mathcal C$ be a class of algebras of the same type and $\mathcal W$ a variety with equationally definable constant $1$. Then the class $\mathcal C$ is called
	\begin{itemize}
		\item {\em congruence permutable} if $\Theta\circ\Phi=\Phi\circ\Theta$ for all $\mathbf A\in\mathcal C$ and $\Theta,\Phi\in\Con\mathbf A$,
		\item {\em congruence distributive} if $(\Theta\lor\Phi)\land\Psi=(\Theta\land\Psi)\lor(\Phi\land\Psi)$ for all $\mathbf A\in\mathcal C$ and $\Theta,\Phi,\Psi\in\Con\mathbf A$,
		\item {\em arithmetical} if it is both congruence permutable and congruence distributive,
		\item {\em weakly regular} if for each $\mathbf A=(A,F)\in\mathcal C$ and all $\Theta,\Phi\in\Con\mathbf A$ with $[1]\Theta=[1]\Phi$ we have $\Theta=\Phi$.
	\end{itemize}
	The following is well-known (cf.\ \cite{CEL}, Theorem~3.1.8, Corollary~3.2.4 and Theorem~6.4.3):
	\begin{itemize}
		\item The class $\mathcal C$ is congruence permutable if there exists a so-called {\em Maltsev term}, i.e.\ a ternary term $p$ satisfying
		\[
		p(x,x,y)\approx p(y,x,x)\approx y,
		\]
		\item The class $\mathcal C$ is congruence distributive if there exists a so-called {\em majority term}, i.e.\ a ternary term $m$ satisfying
		\[
		m(x,x,y)\approx m(x,y,x)\approx m(y,x,x)\approx x,
		\]
		\item The variety $\mathcal W$ is weakly regular if and only if there exists a positive integer $n$ and binary terms $t_1,\ldots,t_n$ such that
		\[
		t_1(x,y)=\cdots=t_n(x,y)=1\text{ if and only if }x=y.
		\]
	\end{itemize}
	
	We are going to show that the variety $\mathcal V$ of lattice skew Hilbert algebras satisfies very strong congruence properties.
	
	\begin{theorem}\label{th7}
		The variety $\mathcal V$ of lattice skew Hilbert algebras is arithmetical and weakly regular.
	\end{theorem}
	
	\begin{proof}
		Since the underlying posets are lattices, $\mathcal V$ is congruence distributive. Now put
		\[
		p(x,y,z):=((x*y)*z)\land((z*y)*x).
		\]
		By (L2) we have $z\leq(z*x)*x$ and $x\leq(x*z)*z$ and hence
		\begin{align*}
			p(x,x,z) & \approx((x*x)*z)\land((z*x)*x)\approx(1*z)\land((z*x)*x)\approx z\land((z*x)*x)\approx z, \\
			p(x,z,z) & \approx((x*z)*z)\land((z*z)*x)\approx((x*z)*z)\land(1*x)\approx((x*z)*z)\land x\approx x,
		\end{align*}
		i.e.\ $p$ is a Maltsev term which means that $\mathcal V$ is also congruence permutable and therefore arithmetical. For weak regularity, consider the binary terms
		\begin{align*}
			t_1(x,y) & :=x*y, \\
			t_2(x,y) & :=y*x.
		\end{align*}
		Clearly, $t_1(x,x)\approx t_2(x,x)\approx1$. Conversely, $t_1(x,y)=t_2(x,y)=1$ implies $x\leq y\leq x$ and therefore $x=y$.
	\end{proof}
	
	Weak regularity means that every congruence $\Theta$ on a lattice skew Hilbert algebra $\mathbf L$ is fully determined by its {\em kernel}, i.e.\ the congruence class $[1]\Theta$. Since $\Theta$ is also a lattice congruence, every class of it is a convex subset of $\mathbf L$. Hence our first task is to describe these classes. For this purpose we introduce the following concept:
	
	\begin{definition}\label{def1}
		Let $\mathbf L=(L,\lor,\land,*,1)$ be a lattice skew Hilbert algebra. A {\em filter} of $\mathbf L$ is a subset $F$ of $L$ containing $1$ such that $x*y,y*x,z*v,v*z\in F$ implies
		\[
		(x\lor z)*(y\lor v),(x\land z)*(y\land v),(x*z)*(y*v)\in F.
		\]
		Let $\Fil\mathbf L$ denote the set of all filters of $\mathbf L$. For any subset $M$ of $L$ define a binary relation $\Phi(M)$ on $L$ as follows:
		\[
		\Phi(M):=\{(x,y)\in L^2\mid x*y,y*x\in M\}.
		\]
	\end{definition}
	
	The relationship between congruences and filters in lattice skew Hilbert algebras is illuminated in the next two theorems.
	
	\begin{theorem}\label{th1}
		Let $\mathbf L=(L,\lor,\land,*,1)$ be a lattice skew Hilbert algebra and $\Theta\in\Con\mathbf L$. Then $[1]\Theta\in\Fil\mathbf L$ and for any $x,y\in L$,
		\[
		(x,y)\in\Theta\text{ if and only if }x*y,y*x\in[1]\Theta,
		\]
		i.e.\ $\Phi([1]\Theta)=\Theta$.
	\end{theorem}
	
	\begin{proof}
		Let $a,b\in L$. If $(a,b)\in\Theta$ then $a*b,b*a\in[a*a]\Theta=[1]\Theta$, i.e.\ $(a,b)\in\Phi([1]\Theta)$.
		Conversely, if $(a,b)\in\Phi([1]\Theta)$ then $a*b,b*a\in[1]\Theta$ and hence, using (1) and (L2),
		\[
		a=a\land((a*b)*b)\mathrel\Theta(1*a)\land(1*b)\mathrel\Theta((b*a)*a)\land b=b,
		\]
		i.e.\ $(a,b)\in\Theta$. This shows $\Phi([1]\Theta)=\Theta$. Due to the substitution property of $\Theta$ with respect to $\lor$, $\land$ and $*$ we see that $[1]\Theta$ satisfies the conditions from Definition~\ref{def1} and hence $[1]\Theta\in\Fil\mathbf L$.
	\end{proof}
	
	Theorem~\ref{th1} witnesses that lattice skew Hilbert algebras are weakly regular.
	
	We can prove also the converse.
	
	\begin{theorem}\label{th2}
		Let $\mathbf L=(L,\lor,\land,*,1)$ be a lattice skew Hilbert algebra and $F\in\Fil\mathbf L$. Then $\Phi(F)\in\Con\mathbf L$ and $[1](\Phi(F))=F$.
	\end{theorem}
	
	\begin{proof}
		Let $a,b,c,d\in L$. Evidently, $\Phi(F)$ is symmetric and since $1\in F$ and $x*x\approx1$ by (L1), it is also reflexive. Assume $a*b,b*a,c*d,d*c\in F$. Then by Definition~\ref{def1}
		\[
		(a\lor c)*(b\lor d),(b\lor d)*(a\lor c),(a\land c)*(b\land d),(b\land d)*(a\land c),(a*c)*(b*d),(b*d)*(a*c)\in F
		\]
		whence
		\[
		(a\lor c,b\lor d),(a\land c,b\land d),(a*c,b*d)\in\Phi(F).
		\]
		Hence $\Phi(F)$ has the substitution property with respect to all basic operations of $\mathbf L$. Since the variety $\mathcal V$ is congruence permutable, $\Phi(F)$ is also transitive, see e.g.\ Werner's Theorem (\cite W) or Corollary~3.1.13 in \cite{CEL}, and hence $\Phi(F)\in\Con\mathbf L$. Finally, the following are equivalent:
		\begin{align*}
			a & \in[1](\Phi(F)), \\
			(a,1) & \in\Phi(F), \\
			a*1,1*a & \in F, \\
			1,a & \in F, \\
			a & \in F
		\end{align*}
		and hence $[1](\Phi(F))=F$.
	\end{proof}
	
	It is elementary to check that for every lattice skew Hilbert algebra $\mathbf L$, $(\Fil\mathbf L,\subseteq)$ is a complete lattice.
	
	The following corollary follows from Theorems~\ref{th1} and \ref{th2}.
	
	\begin{corollary}\label{cor1}
		For every lattice skew Hilbert algebra $\mathbf L$ the mappings $\Phi\mapsto[1]\Phi$ and $F\mapsto\Phi(F)$ are mutually inverse isomorphisms between the complete lattices $(\Con\mathbf L,\subseteq)$ and $(\Fil\mathbf L,\subseteq)$.
	\end{corollary}
	
	Since the operation $*$ may serve as implication in the logic based on a skew Hilbert algebra, we can consider also corresponding deductions. For this reason we introduce the following concept.
	
	A {\em deductive system} of a skew Hilbert algebra $(S,\leq,*,1)$ is a subset $D$ of $S$ containing $1$ and satisfying the following condition:
	\[
	\text{if }a\in D,b\in S\text{ and }a*b\in D\text{ then }b\in D.
	\]
	In the following $(F*(F*a))*a$ denotes the set $\{(x*(y*a))*a\mid x,y\in F\}$. Analogously, we proceed in similar cases.
	
	\begin{theorem}
		Let $\mathbf L=(L,\lor,\land,*)$ be a lattice skew Hilbert algebra, $\Theta\in\Con\mathbf L$, $F\in\Fil\mathbf L$ and $a,b\in L$. Then
		\begin{enumerate}[{\rm(i)}]
			\item Every class of $\Theta$ is a convex subset of $(L,\leq)$,
			\item $F$ is a deductive system of $\mathbf L$,
			\item $F$ is a lattice filter of $\mathbf L$,
			\item $a*(F\land a)\subseteq F$ and $(F*(F*a))*a\subseteq F$.
		\end{enumerate}
	\end{theorem}
	
	\begin{proof}
		\
		\begin{enumerate}[(i)]
			\item If $c,d\in[a]\Theta$ and $c\leq b\leq d$ then
			\[
			b=c\lor b\in[d\lor b]\Theta=[d]\Theta=[a]\Theta.
			\]
			\item If $a,a*b\in F$ then
			\[
			b=1*b\in[a*b](\Phi(F))=[1](\Phi(F))=F.
			\]
			\item If $a\in F$ then
			\[
			a\lor b\in[1\lor b](\Phi(F))=[1](\Phi(F))=F.
			\]
			Moreover, if $a,b\in F$ then
			\[
			a\land b\in[1\land1](\Phi(F))=[1](\Phi(F))=F.
			\]
			\item
			\begin{align*}
				a*(F\land a) & \subseteq[a*(1\land a)](\Phi(F))=[a*a](\Phi(F))=[1](\Phi(F))=F, \\
				(F*(F*a))*a & \subseteq[(1*(1*a))*a](\Phi(F))=[(1*a)*a](\Phi(F))=[a*a](\Phi(F))= \\
				& =[1](\Phi(F))=F.
			\end{align*}
		\end{enumerate}
	\end{proof}
	
	In what follows, we consider congruences in non-lattice skew Hilbert algebras.
	
	Non-lattice skew Hilbert algebras have only one everywhere defined operation, namely $*$. However, the concept of a congruence should respect also the partial order relation. Hence we define
	
	\begin{definition}
		A binary relation $\rho$ on a poset $(P,\leq)$ is called {\em $\min$-stable} if whenever $(a,b),(c,d)\in\rho$, $a$ and $c$ are comparable with each other and $b$ and $d$ are comparable with each other then
		\[
		(\min(a,c),\min(b,d))\in\rho.
		\]
	\end{definition}
	
	Now we define a congruence on a skew Hilbert algebra as follows:
	
	\begin{definition}
		Let $\mathbf S=(S,\leq,*,1)$ be a skew Hilbert algebra. Then
		\begin{itemize}
				\item An {\em algebraic congruence} on $\mathbf S$ is a congruence on its reduct $(S,*)$. 
				\item A {\em congruence} on $\mathbf S$ is a $\min$-stable algebraic congruence on $\mathbf S$. 
		\end{itemize}
		Let $\Con\mathbf S$ denote the set of all congruences on $\mathbf S$.
	\end{definition}
	
	% \textcolor{black}{Thus, we can naturally define:}
	% \begin{definition}
	% \textcolor{black}{By a congruence $\Theta$ on a lattice skew Hilbert algebra $\mathbf L=(L,\lor,\land,*,1)$ we mean
	% \begin{itemize}
	% \item $\Theta \in \Con\mathbf L$, 
	% \item $\Theta$ is a lattice congruence.
	% \end{itemize}}
	% To avoid possible confusion, we call it an {\em $l$-congruence}.
	% \end{definition}

	%Obviously, every congruence on a lattice skew Hilbert algebra is $\min$-stable.
	
	\begin{remark}
		Note that any congruence on a lattice skew Hilbert algebra $\mathbf L=(L,\lor,\land,*,1)$ is automatically a congruence on the underlying skew Hilbert algebra $\mathbf L=(L,\leq,*,1)$.
		
		However, a congruence on a skew Hilbert algebra that is a lattice may not be a congruence on the underlying lattice (see the following counter-example).
	\end{remark}
	
	\begin{example}
		If $(L,\lor,\land)$ denotes the lattice considered in Example~\ref{ex1} {\rm(a)}, we define a binary operation $*$ on $L$ by
		\[
		\begin{array}{c|ccccccc}
			* & 0 & a & b & c & d & e & 1 \\
			\hline
			0 & 1 & 1 & 1 & 1 & 1 & 1 & 1 \\
			a & 0 & 1 & b & 1 & d & e & 1 \\
			b & 0 & a & 1 & 1 & 1 & 1 & 1 \\
			c & 0 & a & b & 1 & d & e & 1 \\
			d & 0 & a & e & c & 1 & e & 1 \\
			e & 0 & a & d & c & d & 1 & 1 \\
			1 & 0 & a & b & c & d & e & 1
		\end{array}
		\]
		and we put
		\[
		\Theta:=\{0\}^2\cup\{a\}^2\cup\{b,e\}^2\cup\{c,d,1\}^2
		\]
		then $\mathbf S:=(L,\leq,*,1)$ is a lattice skew Hilbert algebra, $\Theta\in\Con\mathbf S$, but $\Theta\notin\Con(L,\lor,\land)$ since
		\[
		(c,d)\in\Theta\text{, but }(c\land a,d\land a)=(a,0)\notin\Theta.
		\]
	\end{example}
	
	Using the $\min$-stability property of congruences in skew Hilbert algebras we can prove:
	
	\begin{theorem}
		Let $\mathbf S=(S,\leq,*,1)$ be a skew Hilbert algebra and $\Theta\in\Con\mathbf S$. Then every class of $\Theta$ is a convex subset of $(S,\leq)$.
	\end{theorem}
	
	\begin{proof}
		If $a,c\in S$, $b,d\in[a]\Theta$ and $b\leq c\leq d$ then by (S2) we obtain 
		\begin{align*}
			& (c*d)*b=1*b=b\leq c\leq(c*b)*b, \\
			& ((c*d)*b,(c*b)*b)\in\Theta
		\end{align*}
		and hence by $\min$-stability of $\Theta$ we have 
		\[
		(b,c)=(\min((c*d)*b,c),\min((c*b)*b,c))\in\Theta,
		\]
		which implies $c\in[b]\Theta=[a]\Theta$.
	\end{proof}

	We now investigate quotients of skew Hilbert algebras and strong skew Hilbert algebras with respect to their congruences.
	
	Let $\mathbf S=(S,\leq,*,1)$ be a skew Hilbert algebra and $\Theta$ be an algebraic congruence on $\bf S$. We define a binary relation $\leq'$ on $S/\Theta$ by
	\[
	[a]\Theta\leq'[b]\Theta\text{ if and only if }[a]\Theta*[b]\Theta=[1]\Theta \,\,\, (a,b\in S).
	\]
	Recall that a {\em poset} $(P,\leq)$ is called up-directed if for any $x,y\in P$ there exists some $z\in P$ with $x,y\leq z$. Hence, every poset having a top element is up-directed.
	
	\begin{definition}\label{strong}
		An algebraic congruence $\Theta$ on a skew Hilbert algebra $\mathbf S=(S,\leq,*,1)$ is called {\em strong} if it satisfies for all $a,b\in S$ the condition
		\[
		[a]\Theta\leq'[b]\Theta\text{ if and only if there exists some }c\in[b]\Theta\text{ with }a\leq c\text{ and }b\leq c.
		\]
		We naturally define the term {\em strong congruence} as being a strong algebraic congruence which is min-stable.
	\end{definition}
	
	\begin{theorem}\label{th10}
		Let $\mathbf S=(S,\leq,*,1)$ be a skew Hilbert algebra, $n\geq1$, $a,a_1,\ldots,a_n,b\in S$ and $\Theta$ be an algebraic congruence on $\bf S$. Then
		\begin{enumerate}
			\item[\rm(i)] $a\leq b$ implies $[a]\Theta\leq'[b]\Theta$.
			\item[\rm(ii)] If $\mathbf S$ is strong, then $\Theta$ is strong.
		\end{enumerate}
		If $\Theta$ is a strong algebraic congruence, then
		\begin{enumerate}
			\item[\rm(iii)] Every class of $\Theta$ is up-directed.
			\item[\rm(iv)] $U([a_1]\Theta,\ldots,[a_n]\Theta)=\{[x]\Theta\mid x\in U(a_1.\ldots,a_n)\}$ in $(S/\Theta,\leq')$.
		\end{enumerate}
		If $\Theta$ is moreover a strong congruence, then
		\begin{enumerate}
			\item[\rm(v)] $(S/\Theta,\leq')$ is a poset.
		\end{enumerate}
	\end{theorem}
	
	\begin{proof}
		\
		\begin{enumerate}
			\item[(i)] If $a\leq b$ then $a*b=1$ whence $a*b\mathrel\Theta1$, i.e.\ $[a]\Theta*[b]\Theta=[a*b]\Theta=[1]\Theta$, thus $[a]\Theta\leq'[b]\Theta$.
			\item[(ii)] Suppose that $\mathbf S$ is a strong skew Hilbert algebra. If $[a]\Theta\leq'[b]\Theta$ then $a*b\mathrel\Theta1$ and hence $a\leq(a*b)*b\in[1*b]\Theta=[b]\Theta$. So one can put $c:=(a*b)*b$ and we have also $b\leq c$. If, conversely, there exists some $c\in[b]\Theta$ with $a\leq c$ then according to (ii) we have $[a]\Theta\leq'[c]\Theta=[b]\Theta$.
			
			\medskip
			
			For the following items we have to assume that $\Theta$ is strong algebraic congruence. 
			\item[(iii)] Let $b,c\in[a]\Theta$. Then $[b]\Theta \leq'[c]\Theta$. Hence there exists some $d\in[c]\Theta=[a]\Theta$ such that $b\leq d$ and $c\leq d$.
			\item[(iv)] Assume $[a]\Theta\in U([a_1]\Theta,\ldots,[a_n]\Theta)$. Since $\Theta$ is strong, for all $i\in\{1,\ldots,n\}$ there exists some $b_i\in[a]\Theta$ with $a_i\leq b_i$. Because of (iv), $([a]\Theta,\leq)$ is up-directed and hence there exists some $c\in[a]\Theta$ with $b_1,\ldots,b_n\leq c$. This shows
			\[
			[a]\Theta=[c]\Theta\in\{[x]\Theta\mid x\in U(a_1,\ldots,a_n)\}.
			\]
			The converse inclusion follows from (ii).
			
			\medskip
			
			Now assume that $\Theta$ is moreover $\min$-stable (i.e.\ $\Theta$ turns into a strong congruence).
			\item[(v)] Obviously, $\leq'$ is reflexive. Now assume $[a]\Theta\leq'[b]\Theta$ and $[b]\Theta\leq'[a]\Theta$. Since  $\Theta$ is strong there exists some $c\in[b]\Theta$ with $a\leq c$. Because of $[c]\Theta=[b]\Theta\leq'[a]\Theta$ there exists some $d\in[a]\Theta$ with $c\leq d$. Since $a\leq c\leq d$, $a,d\in[a]\Theta$ and $([a]\Theta,\leq)$ is convex we conclude $c\in[a]\Theta$. Therefore $[a]\Theta=[c]\Theta=[b]\Theta$ which proves antisymmetry of $\leq'$. Finally, let $c\in S$ and assume $[a]\Theta\leq'[b]\Theta$ and $[b]\Theta\leq'[c]\Theta$. Then, from the fact that $\Theta$ is strong we can find some $e\in[b]\Theta$ with $a\leq e$ and because of $[e]\Theta=[b]\Theta\leq'[c]\Theta$ some $f\in[c]\Theta$ with $e\leq f$. From $a\leq e\leq f$ we have $a\leq f$ which implies $[a]\Theta\leq'[f]\Theta=[c]\Theta$ by (ii), proving transitivity of $\leq'$.
		\end{enumerate}
	\end{proof}
	
	From (iv) we conclude that if $(S,\leq)$ satisfies the Ascending Chain Condition (in particular, if $S$ is finite) then every class 
	of a strong congruence $\Theta$ has a greatest element.
	
	However this is not true in general (see the following example).
	
	\begin{example}
		Let $\mathbf S=(S,\leq,1)$ denote the poset with top element $1$ visualized in Fig.~6
		
		\vspace*{10mm}
		
		\begin{center}
			\setlength{\unitlength}{7mm}
			\begin{picture}(8,6)
				\put(3,1){\circle*{.3}}
				\put(5,1){\circle*{.3}}
				\put(1,3){\circle*{.3}}
				\put(3,3){\circle*{.3}}
				\put(5,3){\circle*{.3}}
				\put(5,5){\circle*{.3}}
				\put(3,7){\circle*{.3}}
				\put(7,3){\circle*{.3}}
				\put(3,1){\line(-1,1)2}
				\put(3,1){\line(0,1)6}
				\put(3,1){\line(1,1)2}
				\put(5,3){\line(0,1)2}
				\put(5,1){\line(1,1)2}
				\put(1,3){\line(1,2)2}
				\put(7,3){\line(-1,1)4}
				\put(5,1){\line(-2,1)4}
				\put(2.85,.25){$a$}
				\put(4.85,.25){$b$}
				\put(.4,2.85){$c$}
				\put(3.3,2.85){$d$}
				\put(5.3,2.85){$e$}
				\put(7.3,2.85){$f$}
				\put(4.85,5.5){$g$}
				\put(2.8,7.4){$1$}
				\put(3.2,-0.9){{\rm Fig.~6}}
			\end{picture}
		\end{center}
		
		\vspace*{6mm}
		
		and $*$ the binary operation on $S$ defined by
		\[
		\begin{array}{c|cccccccc}
			* & a & b & c & d & e & f & g & 1 \\
			\hline
			a & 1 & f & 1 & 1 & 1 & f & 1 & 1 \\
			b & e & 1 & 1 & d & e & 1 & 1 & 1 \\
			c & a & b & 1 & d & e & f & g & 1 \\
			d & a & b & c & 1 & e & f & g & 1 \\
			e & a & b & c & d & 1 & f & 1 & 1 \\
			f & a & b & c & d & e & 1 & 1 & 1  \\
			g & a & b & c & d & e & f & 1 & 1 \\
			1 & a & b & c & d & e & f & g & 1
		\end{array}
		\]
		Then $(S,\leq,*,1)$ is a skew Hilbert algebra which is not strong since
		\[
		a\not\leq b=f*b=(a*b)*b.
		\]
		Moreover,
		\[
		\Theta:=\{a,b\}^2\cup\{c\}^2\cup\{d,e,f,g,1\}^2
		\]
		is a congruence on $\mathbf S$ and $([a]\Theta,\leq)=(\{a,b\},\leq)$ has no greatest element. From {\rm(iv)} of Theorem~\ref{th10} we conclude that $\Theta$ is not strong.
	\end{example}
	
	\begin{theorem}
		Let $\mathbf S=(S,\leq,*,1)$ be a skew Hilbert algebra and $\Theta$ an algebraic congruence on $\mathbf S$ such that every class of $\Theta$ satisfies the Ascending Chain Condition. Then $\Theta\in\Con\mathbf S$.
	\end{theorem}
	
	\begin{proof}
		Suppose $(a,b),(c,d)\in\Theta$ where $a$ and $c$ are comparable with each other and $b$ and $d$ are comparable with each other. We have the following four possibilities:
		\begin{itemize}
			\item[(a)] $a\leq c$ and $b\leq d$,
			\item[(b)] $c\leq a$ and $d\leq b$,
			\item[(c)] $a\leq c$ and $d\leq b$,
			\item[(d)] $c\leq a$ and $b\leq d$.
		\end{itemize}
		It is evident that in the first two cases $(\min(a,c),\min(b,d))\in\Theta$. \\
		Now consider case (c). \\
		We have $a\leq c$ and $d\leq b$. Put $a_0:=a$ and $c_0:=c$. Then $a_0\mathrel\Theta a$, $c_0\mathrel\Theta c$ and $c_0\geq a_0$. If $c_0=a_0$ then $a=a_0=c_0=c\mathrel\Theta d$ and we are done. Otherwise we have $c_0>a_0$. Put $a_1:=(c_0*a_0)*a_0$. Then
		\[
		a_1\mathrel\Theta(c*a)*a\mathrel\Theta(d*b)*a=1*a=a
		\]
		and $a_1\geq c_0$ according to (S2). If $a_1=c_0$ then $a\mathrel\Theta a_1=c_0=c\mathrel\Theta d$ and we are done. Otherwise we have $a_1>c_0$. Put $c_1:=(a_1*c_0)*c_0$. Then
		\[
		c_1\mathrel\Theta(a*c)*c=1*c=c
		\]
		and $c_1\geq a_1$ according to (S2). If $c_1=a_1$ then $a\mathrel\Theta a_1=c_1\mathrel\Theta c\mathrel\Theta d$ and we are done. Otherwise we have $c_1>a_1$. Put $a_2:=(c_1*a_1)*a_1$. Then
		\[
		a_2\mathrel\Theta(c*a)*a\mathrel\Theta(d*b)*a=1*a=a
		\]
		and $a_2\geq c_1$ according to (S2). By going on in this way we get a chain of the form
		\begin{align*}
			\textrm{(*)} \quad a_0<c_0<a_1<c_1<a_2<\cdots
		\end{align*}
		where $a_0:=a$, $c_0:=c$ and
		\begin{align*}
			a_k & :=(c_{k-1}*a_{k-1})*a_{k-1}, \\
			c_k & :=(a_k*c_{k-1})*c_{k-1}
		\end{align*}
		for $k>0$. Moreover, $a_k\mathrel\Theta a$ and $c_k\mathrel\Theta c$ for $k\geq0$. If the chain (*) would be infinite then
		\[
		a_0<a_1<a_2<\cdots
		\]
		would be an infinite ascending chain in $([a]\Theta,\leq)$ contradicting the assumption that this poset satisfies the Ascending Chain Condition. Hence there exists some $m\geq0$ such that either $c_m=a_m$ or $a_{m+1}=c_m$. In the first case we have
		\[
		a\mathrel\Theta a_m=c_m\mathrel\Theta c\mathrel\Theta d
		\]
		whereas in the second case
		\[
		a\mathrel\Theta a_{m+1}=c_m\mathrel\Theta c\mathrel\Theta d.
		\]
		This shows $a\mathrel\Theta d$ in case (c). Case (d) is symmetric to case (c).
	\end{proof}
	
	From the preceding theorem  we obtain if $(S,\leq)$ satisfies the Ascending Chain Condition (in particular, if $S$ is finite) then every algebraic congruence on $\mathbf S$ is a congruence on $\mathbf S$. Moreover, if $\mathbf S$ is in addition a strong skew Hilbert algebra then every algebraic congruence on $\mathbf S$ is a strong congruence on $\mathbf S$.
	
	% \textcolor{black}{However, when we take lattice skew Hilbert algebras into account we have the following relationship with strong congruences.}
	
	% \begin{lemma}
	% 	\textcolor{black}{Let $\mathbf{L}=(L,\lor,\land,*,1)$ be a lattice skew Hilbert algebra. Any $l$-congruence $\Theta$ on $\mathbf L$ is strong.}
	% \end{lemma}
	
	% \begin{proof}
	% 	\textcolor{black}{We already know that $(a \lor b) \land (a*b)=b$ and $((a \lor b),(a \lor b)) \in \Theta$. We compute:
	% 	\begin{align*}
	% 	[a]\Theta \leq' [b]\Theta \Leftrightarrow [a]\Theta * [b]\Theta = [1]\Theta & \Leftrightarrow [a*b]\Theta \in [1]\Theta \\
	% 	& \Leftrightarrow (a*b,1) \in \Theta \\
	% 	& \Leftrightarrow \big( (a\lor b) \land (a*b) , (a \lor b) \land 1 \big) \in \Theta \\
	% 	& \Leftrightarrow \big(b,(a\lor b) \big) \in \Theta \\
	% 	& \Leftrightarrow (a \lor b) \in [b]\Theta
	% 	\end{align*}
	% 	Let us put $c:=a \lor b$. So we have $a,b \leq c$ such that $c \in [b]\Theta$.}
	
	% 	\textcolor{black}{Conversely, let $a,b \leq c$ and $c\in[b]\Theta$ for some $a,b,c \in L$. Correspondingly, we have $(b,c) \in \Theta$. Since $(a,a) \in \Theta$, we get $(a*b,a*c)=(a*b,1) \in \Theta$ from (S1) so that $[a]\Theta \leq' [b]\Theta$ by definition. }
	% \end{proof}

	We are now going to show that although the class of strong skew Hilbert algebras does not form a variety, every of its members is weakly regular.
	
	In analogy to the lattice case we define:
	
	\begin{definition}\label{def2}
		Let $\mathbf S=(S,\leq,*,1)$ be a skew Hilbert algebra and $F$ a subset of $S$ containing $1$. We say that
		\begin{enumerate}[{\rm(a)}] 
			\item $F$ is a {\em $*$-filter} of $\mathbf S$ if it satisfies the following condition for all $x,y,z,v\in S$:
			\begin{itemize}
				\item[{\rm(F1)}] if $x*y,y*x,z*v,v*z\in F$ then $(x*z)*(y*v) \in F$.
			\end{itemize}
			\item $F$ is a {\em filter} of $\mathbf S$ if it is a $*$-filter of $\mathbf S$ and satisfies the following condition for all $x,y,z,v\in S$:
			\begin{itemize}
				\item[{\rm(F2)}] if $x*y,y*x,z*v,v*z\in F$, $x$ and $z$ are comparable with each other, and $y$ and $v$ are comparable with each other then $\min(x,z)*\min(y,v)\in F$.
			\end{itemize}
			\item $F$ is a {\em strong $*$-filter} of $\mathbf S$ if it is a $*$-filter of $\mathbf S$ and satisfies the following condition for all $x,y\in S$:
			\begin{itemize}
				\item[{\rm(F3)}] if $x*y\in F$ then there exists some $z\in S$ such that $x,y\leq z$ and $z*y\in F$.
			\end{itemize}
			\item F is a {\em strong filter} of $\bf S$ if it is a filter satisfying {\rm (F3)}.
		\end{enumerate}

		Let $\Fil\mathbf S$ denote the set of all filters of $\mathbf S$. It is elementary to check that for every skew Hilbert algebra $\mathbf S$, $(\Con\mathbf S,\subseteq)$ and $(\Fil\mathbf S,\subseteq)$ are complete lattices. For any subset $M$ of $S$ put
		\[
		\Phi(M):=\{(x,y)\in S^2\mid x*y,y*x\in M\}.
		\]
	\end{definition}
	
The relationship between congruences and filters in strong skew Hilbert algebras is magnified by the following theorems and corollaries.
	
	\begin{theorem}\label{th3}
		Let $\mathbf S=(S,\leq,*,1)$ be a skew Hilbert algebra and $\Theta$ a strong congruence on $\mathbf S$. Then
			\[
			(x,y)\in\Theta\text{ if and only if }x*y,y*x\in[1]\Theta,
			\]
			i.e.\ $\Phi([1]\Theta)=\Theta$.
	\end{theorem}
	
	\begin{proof}
			%		It is readily checked that $[1]\Theta$ is a $*$-filter for any skew Hilbert algebra. So we only need to prove that it is strong. Suppose that $x * y \in [1]\Theta$. We have $[x*y]\Theta=[x]\Theta * [y]\Theta = [1]\Theta$ i.e.\ $[x]\Theta \leq' [y]\Theta$. Since $\Theta$ is strong, there exists $c \in [y]\Theta$ such that $x,y \leq c$. Moreover, since $(c,y) , (y,y) \in \Theta$, we obtain $(c*y,y*y)=(c*y,1) \in \Theta$ and hence $c*y \in [1]\Theta$. 
For $a,b\in S$ the following are equivalent:
			\begin{align*}
				(a,b) & \in\Phi([1]\Theta), \\
				a*b,b*a & \in[1]\Theta, \\
				[a]\Theta & \leq'[b]\Theta\leq'[a]\Theta, \\
				[a]\Theta & =[b]\Theta, \\
				(a,b) & \in\Theta.
			\end{align*}
		%\end{enumerate}
	\end{proof}
	
	\begin{corollary}
		Let $\mathbf S=(S,\leq,*,1)$ be a skew Hilbert algebra and $\Theta \in \Con \bf S$. Then
		\begin{enumerate}[{\rm(a)}] 
			\item If $\Theta$ is strong, then $[1]\Theta$ is a strong filter of $\bf S$.
			% \item If $\mathbf S$ is strong, then $[1]\Theta$ is a strong $*$-filter of $\mathbf S$.
			\item If $\mathbf S$ is strong, then $[1]\Theta$ is a strong filter of $\bf S$.
		\end{enumerate}
	\end{corollary}
	
	\begin{proof}
		Let $a,b,c,d\in S$.
		\begin{enumerate}[(a)]
			\item (F2) Assume $a*b,b*a,c*d,d*c\in[1]\Theta$, $a$ and $c$ are comparable with each other and $b$ and $d$ are comparable with each other. Then $(a,b),(c,d)\in\Theta$ according to Theorem~\ref{th3}. Because of the $\min$-stability of $\Theta$ we have
			\[
			(\min(a,c),\min(b,d))\in\Theta
			\]
			whence
			\[
			\min(a,c)*\min(b,d)\in[\min(a,c)*\min(a,c)]\Theta=[1]\Theta.
			\]
			(F3) Follows from Theorem \ref{th3}.
			% \item Since $\mathbf S$ is strong we obtain that $\Theta$ is strong as well according to Theorem~\ref{th10}. Hence from Theorem~\ref{th3} we have that $[1]\Theta$ is a strong $*$-filter of $\mathbf S$.
			\item This follows from (a) and Theorem~\ref{th10} (ii).
		\end{enumerate}
	\end{proof}
	
	We have shown that every congruence $\Theta$ on a strong skew Hilbert algebra is fully determined by its $1$-class $[1]\Theta$. Hence we conclude
	
	\begin{corollary}
		Strong skew Hilbert algebras are weakly regular.
	\end{corollary}
	
	\begin{proof}
		Let $\mathbf S=(S,\leq,*,1)$ be a strong skew Hilbert algebra and $\Theta, \Phi \Con\mathbf S$. Thus $\Theta, \Phi$ is strong from Theorem~\ref{th10}. Suppose that $[1]\Theta=[1]\Phi$. Then considering (ii) of Theorem \ref{th3}, we obtain:
			\begin{align*}
				(a,b) \in \Theta \Leftrightarrow a*b,b*a \in [1]\Theta \Leftrightarrow a*b,b*a \in [1]\Phi \Leftrightarrow (a,b) \in \Phi.
		\end{align*} 
	\end{proof}
	
	The preceding corollary is not true in general case (see the following example).
	
	\begin{example}
		Let $\mathbf S=(S,\leq,1)$ denote the poset with top element $1$ visualized in Fig.~7
		
		\vspace*{10mm}
		
		\begin{center}
			\setlength{\unitlength}{7mm}
			\begin{picture}(8,6)
				\put(3,1){\circle*{.3}}
				\put(3,3){\circle*{.3}}
				\put(3,5){\circle*{.3}}
				\put(5,1){\circle*{.3}}
				\put(3,3){\circle*{.3}}
				\put(5,3){\circle*{.3}}
				\put(3,7){\circle*{.3}}
				\put(1,3){\circle*{.3}}
				\put(3,1){\line(0,1)6}
				\put(1,3){\line(1,2)2}
				\put(3,1){\line(-1,1)2}
				\put(5,3){\line(-1,2)2}
				\put(5,1){\line(0,1)2}
				\put(5,1){\line(-1,1)2}
				\put(2.85,.25){$a$}
				\put(4.85,.25){$b$}
				\put(3.3,2.85){$c$}
				\put(3.3,4.85){$f$}
				\put(5.3,2.85){$d$}
				\put(0.3,2.85){$e$}
				\put(2.8,7.4){$1$}
				\put(3.2,-0.9){{\rm Fig.~7}}
			\end{picture}
		\end{center}
		
		\vspace*{6mm}
		
		and $*$ the binary operation on $S$ defined by
		\[
		\begin{array}{c|ccccccc}
			* & a & b & c & d & e & f & 1 \\
			\hline
			a & 1 &d&1&d&1&1&1 \\
			b &e&1&1&1&e&1&1 \\
			c &a&b&1&d&e&1& 1 \\
			d &a&b&c&1&e&f&1 \\
			e &a&b&c&d&1&f&1 \\
			f &a&b&c&d&e&1&1  \\
			1 & a & b & c & d & e &f & 1
		\end{array}
		\]
		Then $(S,\leq,*,1)$ is a skew Hilbert algebra which is not strong since 
			\[
			a\not\leq b=d*b=(a*b)*b.
			\]
			We have a congruence given by
			\[
			\Theta:=\{a\}^2 \cup \{b\}^2 \cup \{c\}^2\cup\{d,e,f,1\}^2.
			\]
			It is readily checked that $[1]\Theta$ is a (strong) filter. However, $a*b,b*a \in [1]\Theta$ but $(a,b) \not\in \Theta$.
	\end{example}
	
	We can prove also the converse of Theorem~\ref{th3}.
	
	\begin{theorem}\label{th4}
		Let $\mathbf S=(S,\leq,*,1)$ be a skew Hilbert algebra and $F$ a $*$-filter of $\mathbf S$. Then $\Phi(F)$ is an algebraic congruence on $\bf S$ and $[1](\Phi(F))=F$.
	\end{theorem}
	
	\begin{proof}
		Let $a,b,c,d\in S$. Evidently, $\Phi(F)$ is symmetric and since $1\in F$ and $x*x\approx1$, it is also reflexive. Let us show that $(a,b),(c,d)\in\Phi(F)$ implies $(a*c,b*d)\in\Phi(F)$. We have $a*b,b*a,c*d,d*c\in F$ by definition. From (F1) we get
		\begin{align*}
			& (a*c)*(b*d)\in F , \\
			& (b*d)*(a*c)\in F
		\end{align*} 
		which yields $(a*c,b*d)\in\Phi(F)$ by definition. Hence $\Phi(F)$ has the substitution property with respect to $*$. If $(a,b),(b,c)\in\Phi(F)$ then
		\begin{align*}
			a*c & =1*(a*c)=(b*b)*(a*c)\in F, \\
			c*a & =1*(c*a)=(b*b)*(c*a)\in F
		\end{align*}
		and hence $(a,c)\in\Phi(F)$. This shows transitivity of $\Phi(F)$. Finally, the following are equivalent:
		\begin{align*}
			a & \in[1](\Phi(F)), \\
			(a,1) & \in\Phi(F), \\
			a*1,1*a & \in F, \\
			1,a & \in F, \\
			a & \in F
		\end{align*}
		showing $[1](\Phi(F))=F$.
	\end{proof}
	
	\begin{corollary}
		Let $\mathbf S=(S,\leq,*,1)$ be a skew Hilbert algebra and $F$ a $*$-filter of $\mathbf S$. Then
		\begin{enumerate}[{\rm(a)}] 
			\item If $F\in\Fil\mathbf S$ then $\Phi(F)\in\Con\mathbf S$.
			\item If $F$ is strong then $\Phi(F)$ is strong.
		\end{enumerate}
	\end{corollary}
	
	\begin{proof}
		Let $a,b,c,d\in S$.
		\begin{enumerate}[(a)]
			\item If $(a,b),(c,d)\in\Phi(F)$, $a$ and $c$ are comparable with each other and $b$ and $d$ are comparable with each other then from (F2) we get
			\[
			\min(a,c)*\min(b,d),\min(b,d)*\min(a,c)\in F,
			\]
			i.e.\ $(\min(a,c),\min(b,d))\in\Phi(F)$. This shows that $\Phi(F)$ is $\min$-stable.
			\item Assume $[a]\Phi(F)\leq'[b]\Phi(F)$. Then $(a*b,1)\in\Phi(F)$, i.e.\ $a*b\in F$. From (F3) we get that there exists some $c\in S$ such that $a,b\leq c$ and $c*b\in F$. Since $1=b*c\in F$ we obtain $c\in[b]\Phi(F)$, i.e.\ $\Phi(F)$ is strong.
		\end{enumerate}
	\end{proof}
	
	The following corollary follows from the above theorems and corollaries.
	
	\begin{corollary}
		For every strong skew Hilbert algebra $\mathbf S$ the mappings $\Phi\mapsto[1]\Phi$ and $F\mapsto\Phi(F)$ are mutually inverse isomorphisms between the complete lattices $(\Con\mathbf S,\subseteq)$ and $(\Fil\mathbf S,\subseteq)$.
	\end{corollary}
	
	We can prove the following result in analogy to the corresponding result for lattice skew Hilbert algebras.
	
	\begin{theorem}\label{th5}
		Let $\mathbf S=(S,\leq,*,1)$ be a skew Hilbert algebra, $F\in\Fil\mathbf S$ and $c\in S$. Then
		\begin{enumerate}[{\rm(i)}]
			\item $F$ is a deductive system of $\mathbf S$,
			\item $F$ is an order filter of $\mathbf S$,
			\item $S*F\subseteq F$,
			\item $c*(F\land c),(F*(F*c))*c\subseteq F$.
		\end{enumerate}
	\end{theorem}
	
	\begin{proof}
		We use the fact that the filter $F$ is the $1$-class of the congruence $\Phi(F)$.
		\begin{enumerate}[(i)]
			\item If $a\in F$, $b\in S$ and $a*b\in F$ then
			\[
			b=1*b\in[a*b](\Phi(F))=[1](\Phi(F))=F.
			\]
			\item If $a\in F$, $b\in S$ and $a\leq b$ then $a*b=1\in F$ and hence $b\in F$ by (i).
			\item If $a\in S$ and $b\in F$ then $a*b\in[a*1](\Phi(F))=[1](\Phi(F))=F$.
			\item
			\begin{align*}
				c*(F\land c) & \subseteq[c*(1\land c)](\Phi(F))=[c*c](\Phi(F))=[1](\Phi(F))=F, \\
				(F*(F*c))*c & \subseteq[(1*(1*c))*c](\Phi(F))=[(1*c)*c](\Phi(F))=[c*c](\Phi(F))= \\
				& =[1](\Phi(F))=F.
			\end{align*}
		\end{enumerate}
	\end{proof}
	
	Theorem~\ref{th5} shows that every filter is a deductive system. However, our concept of a filter is rather complicated and it seems that not all the properties of a filter are necessary to prove this assertion. We can prove
	
	\begin{proposition}\label{prop4}
		Let $\mathbf S=(S,\leq,*,1)$ be a strong skew Hilbert algebra and $M$ a subset of $S$ containing $1$ and satisfying $(M*(M*x))*x\subseteq M$ for all $x\in S$. Then $M$ is a deductive system of $\mathbf S$.
	\end{proposition}
	
	\begin{proof}
		Let $a\in M$ and $b\in S$. We have $1\in M$. If $a\leq b$ then
		\[
		b=1*b=(a*b)*b=(a*(1*b))*b\in(M*(M*b))*b\subseteq M.
		\]
		Hence, if $a*b\in M$ then because of $a\leq(a*b)*b$ we have $(a*b)*b\in M$ which implies
		\[
		b=1*b=(((a*b)*b)((a*b)*b))*b\in(M*(M*b))*b\subseteq M.
		\]
	\end{proof}
	
	Observe that the condition mentioned in Proposition~\ref{prop4} is just the second one from (iv) of Theorem~\ref{th5}.
	
	For the concept of an ideal of a universal algebra which corresponds to our concept of a filter and for the concept of ideal terms the reader is referred to \cite U. In particular, for ideals (alias filters) in permutable and weakly regular varieties see also \cite{CEL} for details.
	
	\begin{definition}
		An {\em ideal term} for lattice skew Hilbert algebras is a term $t(x_1,\ldots,x_n,y_1,$ $\ldots,y_m)$ in the language of lattice skew Hilbert algebras satisfying the identity
		\[
		t(x_1,\ldots,x_n,1,\ldots,1)\approx1.
		\]
	\end{definition}
	
	Of course, there exists an infinite number of ideal terms in lattice skew Hilbert algebras. The following list including four ideal terms is a so-called {\em basis for filters} in lattice skew Hilbert algebras, i.e.\ filters can be characterized by this short list of ideal terms.
	
	Consider the following terms for lattice skew Hilbert algebras $(L,\lor,\land,*,1)$:
	\begin{align*}
		t(x,y,z,u) & :=(x\lor y)\land(z*y)\land u, \\
		t_1 & :=1, \\
		t_2(x_1,x_2,x_3,x_4,y_1,y_2,y_3,y_4) & :=(t(x_1,x_2,y_1,y_2)\lor t(x_3,x_4,y_3,y_4))*(x_2\lor x_4), \\
		t_3(x_1,x_2,x_3,x_4,y_1,y_2,y_3,y_4) & :=(t(x_1,x_2,y_1,y_2)\land t(x_3,x_4,y_3,y_4))*(x_2\land x_4), \\
		t_4(x_1,x_2,x_3,x_4,y_1,y_2,y_3,y_4) & :=(t(x_1,x_2,y_1,y_2)*t(x_3,x_4,y_3,y_4))*(x_2*x_4).
	\end{align*}
	
	\begin{lemma}
		We have
		\begin{align*}
			t(x,y,1,1) & \approx y, \\
			t(x,y,x*y,y*x) & \approx x.
		\end{align*}
	\end{lemma}
	
	\begin{proof}
		We have
		\[
		t(x,y,1,1)=(x\lor y)\land(1*y)\land1=(x\lor y)\land y=y
		\]
		according to (1) and
		\begin{align*}
			t(x,y,x*y,y*x) & =(x\lor y)\land((x*y)*y)\land(y*x)= \\
			& =((y\lor x)\land(y*x))\land((x*y)*y)=x\land((x*y)*y)=x
		\end{align*}
		according to (L2) and (L4).
	\end{proof}
	
	\begin{lemma}\label{lem1}
		The terms $t_1,\ldots,t_4$ are ideal terms for lattice skew Hilbert algebras.
	\end{lemma}
	
	\begin{proof}
		We have
		\begin{align*}
			t_2(x_1,x_2,x_3,x_4,1,1,1,1) & \approx(t(x_1,x_2,1,1)\lor t(x_3,x_4,1,1))*(x_2\lor x_4)\approx \\
			& \approx(x_2\lor x_4)*(x_2\lor x_4)\approx1, \\
			t_3(x_1,x_2,x_3,x_4,1,1,1,1) & \approx(t(x_1,x_2,1,1)\land t(x_3,x_4,1,1))*(x_2\land x_4)\approx \\
			& \approx(x_2\land x_4)*(x_2\land x_4)\approx1, \\
			t_4(x_1,x_2,x_3,x_4,1,1,1,1) & \approx(t(x_1,x_2,1,1)*t(x_3,x_4,1,1))*(x_2*x_4)\approx \\
			& \approx(x_2*x_4)*(x_2*x_4)\approx1.
		\end{align*}
	\end{proof}
	
	The closedness with respect to ideal terms was also introduced by A.~Ursini (\cite U).
	
	\begin{definition}
		A {\em subset} $A$ of a lattice skew Hilbert algebra $\mathbf L=(L,\lor,\land,*,1)$ is said to be {\em closed} with respect to the ideal terms $t_i(x_1,\ldots,x_n,y_1,\ldots,y_m)$, $i\in I$, if for every $i\in I$, all $x_1,\ldots,x_n\in L$ and all $y_1,\ldots,y_m\in A$ we have $t_i(x_1,\ldots,x_n,y_1,\ldots,y_m)\in A$.
	\end{definition}
	
	Now we prove that the ideal terms listed in Lemma~\ref{lem1} form a basis for filters, i.e.\ filters are characterized as those subsets which are closed with respect to these ideal terms.
	
	\begin{theorem}\label{th6}
		Let $\mathbf L=(L,\lor,\land,*,1)$ be a lattice skew Hilbert algebra and $F\subseteq L$. Then $F\in\Fil\mathbf L$ if and only if $F$ is closed with respect to the ideal terms $t_1,\ldots,t_4$ listed in Lemma~\ref{lem1}.
	\end{theorem}
	
	\begin{proof}
		If $F\in\Fil\mathbf L$ then $F=[1](\Phi(F))$ according to Theorem~\ref{th2}, and if
		\[
		t_i(x_1,\ldots,x_n,y_1,\ldots,y_m),\quad i\in\{1,\ldots,5\},
		\]
		are the ideal terms listed in Lemma~\ref{lem1}, $a_1,\ldots,a_n\in L$ and $b_1,\ldots,b_m\in F$ then
		\[
		t_i(a_1,\ldots,a_n,b_1,\ldots,b_m)\in[t_i(a_1,\ldots,a_n,1,\ldots,1)](\Phi(F))=[1](\Phi(F))=F
		\]
		according to Lemma~\ref{lem1} and hence $F$ is closed with respect to the ideal terms $t_1,\ldots,t_5$. Conversely, assume $F$ to be closed with respect to the ideal terms $t_1,\ldots,t_4$. Then $1=t_1\in F$ and if $a,b,c,d\in L$ and $a*b,b*a,c*d,d*c\in F$ then
		\begin{align*}
			(a\lor c)*(b\lor d) & =(t(a,b,a*b,b*a)\lor t(c,d,c*d,d*c))*(b\lor d)= \\
			& =t_2(a,b,c,d,a*b,b*a,c*d,d*c)\in F, \\
			(a\land c)*(b\land d) & =(t(a,b,a*b,b*a)\land t(c,d,c*d,d*c))*(b\land d)= \\
			& =t_3(a,b,c,d,a*b,b*a,c*d,d*c)\in F, \\
			(a*c)*(b*d) & =(t(a,b,a*b,b*a)*t(c,d,c*d,d*c))*(b*d)= \\
			& =t_4(a,b,c,d,a*b,b*a,c*d,d*c)\in F.
		\end{align*}
		showing $F\in\Fil\mathbf L$.
	\end{proof}
	
	\begin{remark}
		Let us note that the term $t$ from the proof of Theorem~\ref{th6} gives rise to a Maltsev term. Namely, if
		\begin{align*}
			t(x,y,z,u) & :=(x\lor y)\land(z*y)\land u\text{ and} \\
			q(x,y,z) & :=t(x,z,x*y,y*x).
		\end{align*}
		then
		\begin{align*}
			q(x,y,z) & =(x\lor z)\land((x*y)*z)\land(y*x), \\
			q(x,x,z) & =(x\lor z)\land((x*x)*z)\land(x*x)=(x\lor z)\land(1*z)\land1=(x\lor z)\land z=z, \\
			q(x,z,z) & =(x\lor z)\land((x*z)*z)\land(z*x)=((z\lor x)\land(z*x))\land((x*z)*z)= \\
			& =x\land((x*z)*z)=x.
		\end{align*}
		Observe that the Maltsev term $q(x,y,z)$ is different from that in Theorem~\ref{th7}.
	\end{remark}
	
	In the following we write $a\land b\land c$ instead of $\inf(a,b,c)$.
	
	Now we introduce a certain modification of the notion an ideal term (for posets) which need not be defined everywhere. This will be used in the sequel.
	
	\begin{definition}
		A {\em partial ideal term} for skew Hilbert algebras is a partially defined term $T(x_1,\ldots,x_n,y_1,\ldots,y_m)$ in the language of skew Hilbert algebras satisfying the identity
		\[
		T(x_1,\ldots,x_n,1,\ldots,1)\approx1.
		\]
		This language contains also a binary operator $U(x,y)$.
	\end{definition}
	
	Using of the concept of partial ideal terms, we will try to describe filters also in strong skew Hilbert algebras. Similarly as in Lemma~\ref{lem1} we first get a list of three partial ideal terms which will be shown to suffice.
	
	Consider the following partial terms for skew Hilbert algebras $(S,\leq,*,1)$:
	\begin{align*}
		T(x,y,z,u) & :=U(x,y)\land(z*y)\land u, \\
		T_1 & :=1, \\
		T_2(x_1,x_2,x_3,x_4,y_1,y_2,y_3,y_4) & :=(T(x_1,x_2,y_1,y_2)*T(x_3,x_4,y_3,y_4))*(x_2*x_4), \\
		T_3(x_1,x_2,x_3,x_4,y_1,y_2,y_3,y_4) & :=(T(x_1,x_2,y_1,y_2)\land T(x_3,x_4,y_3,y_4))*(x_2\land x_4).
	\end{align*}
	
	\begin{lemma}
		We have
		\begin{align*}
			T(x,y,1,1) & \approx y, \\
			T(x,y,x*y,y*x) & \approx x.
		\end{align*}
	\end{lemma}
	
	\begin{proof}
		We have
		\[
		T(x,y,1,1)\approx U(x,y)\land(1*y)\land1\approx U(x,y)\land y\approx y
		\]
		according to (1) and Remark~\ref{rem2} and
		\begin{align*}
			T(x,y,x*y,y*x) & \approx U(x,y)\land((x*y)*y)\land(y*x)\approx \\
			& \approx(U(y,x)\land(y*x))\land((x*y)*y)\approx x\land((x*y)*y)\approx x
		\end{align*}
		according to (S2') and Remark~\ref{rem2}.
	\end{proof}
	
	\begin{lemma}\label{lem2}
		The partial terms $T_1,T_2,T_3$ are partial ideal terms for skew Hilbert algebras.
	\end{lemma}
	
	\begin{proof}
		We have
		\begin{align*}
			T_1 & \approx1, \\
			T_2(x_1,x_2,x_3,x_4,1,1,1,1) & \approx(T(x_1,x_2,1,1)*T(x_3,x_4,1,1))*(x_2*x_4)\approx \\
			& \approx(x_1*x_4)*(x_2*x_4)\approx1, \\
			T_3(x_1,x_2,x_3,x_4,1,1,1,1) & \approx(T(x_1,x_2,1,1)\land T(x_3,x_4,1,1))*(x_2\land x_4)\approx \\
			& \approx(x_2\land x_4)*(x_2\land x_4)\approx1.
		\end{align*}
	\end{proof}
	
	Now we define closedness with respect to partial ideal terms.
	
	\begin{definition}
		A {\em subset} $A$ of a skew Hilbert algebra $\mathbf S=(S,\leq,*,1)$ is said to be {\em closed} with respect to the partial ideal terms $T_i(x_1,\ldots,x_n,y_1,\ldots$ $\ldots,y_m)$, $i\in I$, if for every $i\in I$, all $x_1,\ldots,x_n\in S$ and all $y_1,\ldots,y_m\in A$ we have that $T_i(x_1,\ldots,x_n,y_1,\ldots,y_m)$ is defined and $T_i(x_1,\ldots,x_n,y_1,\ldots,y_m)\in A$.
	\end{definition}
	
	Although our ideal terms are only partial, we can prove that every subset of a strong skew Hilbert algebra $\mathbf S$ closed with respect to them is really a filter of $\mathbf S$.
	
	\begin{theorem}
		Let $\mathbf S=(S,\leq,*,1)$ be a strong skew Hilbert algebra and $F$ a subset of $S$ that is closed with respect to the partial ideal terms $T_1,T_2,T_3$ listed in Lemma~\ref{lem2}. Then $F\in\Fil\mathbf S$.
	\end{theorem}
	
	\begin{proof}
		We have $1=T_1\in F$. Now assume $a,b,c,d\in S$ and $a*b,b*a,c*d,d*c\in F$. Then
		\begin{align*}
			(a*c)*(b*d) & =(T(a,b,a*b,b*a)*T(c,d,c*d,d*c))*(b*d)= \\
			& =T_2(a,b,c,d,a*b,b*a,c*d,d*c)\in F.
		\end{align*}  
		Moreover, if $a$ and $c$ are comparable with each other and $b$ and $d$ are comparable with each other then we apply the partial term $T_3$ to derive
		\begin{align*}
			\min(a,c)*\min(b,d) & =(T(a,b,a*b,b*a)\land T(c,d,c*d,d*c))*(b\land d)= \\
			& =T_3(a,b,c,d,a*b,b*a,c*d,d*c)\in F.
		\end{align*}
		This shows $F\in\Fil\mathbf S$.
	\end{proof}
	
	\begin{remark}
		Let us consider the partial term $T(x,y,z,u):=U(x,y)\land(z*y)\land u$ from the proof of Lemma~\ref{lem2} and put
		\[
		T_1(x,y,z):=T(x,z,x*y,y*x),
		\]
		i.e.
		\[
		T_1(x,y,z)=U(x,z)\land((x*y)*z)\land(y*x).
		\]
		Of course, this is only a partial term because the infimum in $T_1$ need not exist for some elements from a skew Hilbert algebra $(S,\leq,*,1)$. It is of some interest that in the case of strong skew Hilbert algebras this partial term behaves like a Maltsev term. Namely, we can easily compute
		\begin{align*}
			T_1(x,x,z) & =U(x,z)\land((x*x)*z)\land(x*x)=U(x,z)\land(1*z)\land1=U(x,y)\land z=z, \\
			T_1(x,z,z) & =U(x,z)\land((x*z)*z)\land(z*x)=(U(z,x)\land(z*x))\land((x*z)*z)= \\
			& =x\land((x*z)*z)=x.
		\end{align*}
		Moreover, these expressions $T_1(x,x,z)$ and $T_1(x,z,z)$ are defined for all $x,z\in S $.
	\end{remark}
	
	For every lattice skew Hilbert algebra $\mathbf L=(L,\lor,\land,*,1)$ and every $M\subseteq L$ let $F(M)$ denote the filter of $\mathbf L$ generated by $M$.
	
	The connection between filters generated by a certain subset and congruences on lattice skew Hilbert algebras is described in the following proposition.
	
	\begin{proposition}
		Let $(L,\lor,\land,*,1)$ be a lattice skew Hilbert algebra, $M\subseteq L$ and $a\in L$. Then
		\begin{align*}
			\Phi(F(M)) & =\Theta(M\times\{1\}), \\
			[1](\Theta(M\times\{1\})) & =F(M).
		\end{align*}
		In particular,
		\begin{align*}
			\Phi(F(a)) & =\Theta(a,1), \\
			[1](\Theta(a,1)) & =F(a).
		\end{align*}
	\end{proposition}
	
	\begin{proof}
		Since $M\times\{1\}\subseteq\Phi(F(M))$ we have
		\[
		\Theta(M\times\{1\})\subseteq\Phi(F(M))
		\]
		and hence
		\[
		[1](\Theta(M\times\{1\}))\subseteq[1](\Phi(F(M)))=F(M)
		\]
		according to Corollary~\ref{cor1}. Because of $M\subseteq[1](\Theta(M\times\{1\}))$ we have
		\[
		F(M)\subseteq[1](\Theta(M\times\{1\}))
		\]
		and hence
		\[
		\Phi(F(M))\subseteq\Phi([1](\Theta(M\times\{1\})))=\Theta(M\times\{1\})
		\]
		according to Corollary~\ref{cor1}.
	\end{proof}
	
	An analogous result holds for strong skew Hilbert algebras.
	
	\subsection*{Acknowledgements}
	I.~Chajda, K.~Emir, H.~L\"{a}nger, and J.~Paseka gratefully acknowledge support by the Austrian Science Fund (FWF), project I~4579-N, and the Czech Science Foundation (GA\v CR), project 20-09869L, entitled ``The many facets of orthomodularity''. I.~Chajda's and H.~L\"{a}nger's research has been funded by \"OAD, project CZ~02/2019, entitled ``Function algebras and ordered structures related to logic and data fusion''. I.~Chajda gratefully acknowledges support by IGA, project P\v rF~2021~030. D.~Fazio and A.~Ledda gratefully acknowledge support by Regione Autonoma della Sardegna within the project 
	\textquotedblleft Per un'estensione semantica della Logica Computazionale
	Quantistica - Impatto teorico e ricadute implementative\textquotedblright,
	RAS: SR40341, and the support of MIUR within the project PRIN 2017: \textquotedblleft Logic and cognition. Theory, experiments, and applications\textquotedblright, CUP: 2013YP4N3.

Authors' addresses:
	
	Ivan Chajda \\
	Palack\'y University Olomouc \\
	Faculty of Science \\
	Department of Algebra and Geometry \\
	17.\ listopadu 12 \\
	771 46 Olomouc \\
	Czech Republic \\
	ivan.chajda@upol.cz
	
	Kadir Emir \\
	Masaryk University Brno \\
	Faculty of Science \\
	Department of Mathematics and Statistics \\
	Kotl\'a\v rsk\'a 2 \\
	611 37 Brno \\
	Czech Republic \\
	emir@math.muni.cz
	
	Davide Fazio \\
	A.Lo.P.Hi.S Research Group\\
	%Department of Pedagogy, Psychology and Phylosophy\\
	University of Cagliari\\
	Via Is Mirrionis, 1\\
	09123, Cagliari, Italy\\
	dav.faz@hotmail.it
	
	Helmut L\"anger \\
	TU Wien \\
	Faculty of Mathematics and Geoinformation \\
	Institute of Discrete Mathematics and Geometry \\
	Wiedner Hauptstra\ss e 8-10 \\
	1040 Vienna \\
	Austria, and \\
	Palack\'y University Olomouc \\
	Faculty of Science \\
	Department of Algebra and Geometry \\
	17.\ listopadu 12 \\
	771 46 Olomouc \\
	Czech Republic \\
	helmut.laenger@tuwien.ac.at
	
	Antonio Ledda \\
	A.Lo.P.Hi.S Research Group\\
	%Department of Pedagogy, Psychology and Phylosophy\\
	University of Cagliari\\
	Via Is Mirrionis, 1\\
	09123, Cagliari, Italy\\
	antonio.ledda@unica.it
	
	Jan Paseka \\
	Masaryk University Brno \\
	Faculty of Science \\
	Department of Mathematics and Statistics \\
	Kotl\'a\v rsk\'a 2 \\
	611 37 Brno \\
	Czech Republic \\
	paseka@math.muni.cz
\end{document}